\documentclass[3p,preprint]{elsarticle}
\usepackage{amssymb,latexsym,amsmath,amsthm,lineno,color,tikz,siunitx,subfigure}

\journal{}

\newcommand{\eps}{\varepsilon}
\newcommand{\set}[1]{\left\{#1\right\}}
\newcommand{\abs}[1]{\left|#1\right|}
\newcommand{\epsa}{\varepsilon_{\mathrm{a}}}
\newcommand{\epsb}{\varepsilon_{\mathrm{b}}}
\newcommand{\mua}{\mu_{\mathrm{a}}}
\newcommand{\mub}{\mu_{\mathrm{b}}}
\newcommand{\kb}{k_{\mathrm{b}}}
\newcommand{\p}{\partial}

\newcommand{\mE}{\mathbf{E}}
\newcommand{\mF}{\mathbf{F}}

\newcommand{\mU}{\mathbf{U}}
\newcommand{\mV}{\mathbf{V}}

\newcommand{\me}{\mathbf{e}}
\newcommand{\mf}{\mathbf{f}}
\newcommand{\mg}{\mathbf{g}}

\newcommand{\mr}{\mathbf{x}}

\newcommand{\mx}{\mathbf{x}}
\newcommand{\mz}{\mathbf{z}}

\newcommand{\vt}{\boldsymbol{\theta}}
\newcommand{\vv}{\boldsymbol{\vartheta}}
\newcommand{\vx}{\boldsymbol{\xi}}

\DeclareMathOperator*{\inc}{inc}
\DeclareMathOperator*{\scat}{scat}
\DeclareMathOperator*{\noise}{noise}
\DeclareMathOperator*{\tm}{TM}
\DeclareMathOperator*{\te}{TE}
\DeclareMathOperator*{\fm}{FM}
\DeclareMathOperator*{\dm}{DM}
\DeclareMathOperator*{\de}{DE}

\theoremstyle{plain}
\newtheorem{thm}{Theorem}[section]
\newtheorem{cor}[thm]{Corollary}
\newtheorem{lem}[thm]{Lemma}

\theoremstyle{remark}
\newtheorem{rem}{Remark}[section]
\newtheorem{ex}{Example}[section]

\begin{document}

\begin{frontmatter}



\title{Performance analysis of MUSIC-type imaging without diagonal elements of multi-static response matrix}

\author{Won-Kwang Park}
\ead{parkwk@kookmin.ac.kr}
\address{Department of Information Security, Cryptology, and Mathematics, Kookmin University, Seoul, 02707, Korea.}

\begin{abstract}
  Generally, to apply the MUltiple SIgnal Classification (MUSIC) algorithm for the rapid imaging of small inhomogeneities, the complete elements of the multi-static response (MSR) matrix must be collected. However, in real-world applications such as microwave imaging or bistatic measurement configuration, diagonal elements of the MSR matrix are unknown. Nevertheless, it is possible to obtain imaging results using a traditional approach but theoretical reason of the applicability has not been investigated yet. In this paper, we establish mathematical structures of the imaging function of MUSIC from an MSR matrix without diagonal elements in both transverse magnetic (TM) and transverse electric (TE) polarizations. The established structures demonstrate why the shape of the location of small inhomogeneities can be retrieved via MUSIC without the diagonal elements of the MSR matrix. In addition, they reveal the intrinsic properties of imaging and the fundamental limitations. Results of numerical simulations are also provided to support the identified structures.
\end{abstract}


\end{frontmatter}





\section{Introduction}
Time-harmonic inverse scattering problems for the retrieval of a two-dimensional small inhomogeneities in transverse magnetic (TM) polarization (or permittivity contrast case) and transverse electric (TE) polarization (or permeability contrast case) have been considered in various researches \cite{AIL2,AIM,AK2,AMV,B8,DL,SH,ZT}. The principle of retrieving unknown targets is based on the Newton iteration method (i.e., determining the shape of the imhomogeneities), which minimizes the discrepancy function between the measured far-field patterns in the presence of true and man-made targets. Various techniques for reconstructing the shape of targets have also been developed, including the Newton or Gauss-Newton methods \cite{K3,KP2,W2}, level-set strategy \cite{ADIM,DL,PL4,VXB}, factorization method \cite{KG,LLP,P-FAC1}, potential drop method \cite{IYM}, inverse Fourier transform \cite{AS1}, subspace migration \cite{AGKPS,P-SUB11,P-SUB16}, topological derivative \cite{B1,P-TD1,P-TD3}, direct sampling method \cite{IJZ1,KLP1,KLP3}, and linear sampling method \cite{C,CC,KR}.

The MUltiple SIgnal Classification (MUSIC) algorithm has been successfully used for imaging arbitrary shaped targets. For example, identification of two- and three-dimensional small targets \cite{AILP,CA}, retrieving small targets completely embedded in a half-space \cite{AIL1,IGLP,G2,SCC}, detecting internal corrosion \cite{AKKLV}, damage diagnosis on complex aircraft structures \cite{BYG}, reconstruction of thin inhomogeneities or perfectly conducting cracks \cite{AKLP,P-MUSIC1,PL1,PL3}, imaging of extended targets \cite{AGKLS,HSZ1,LH2,MGS}, radar imaging \cite{OBP}, and biomedical imaging \cite{RSAAP,RSCGBA,S2}. We also refer to \cite{CA,CZ,F3,HLD,K1,SCC,SNPM,ZC} for various application of MUSIC algorithm. Throughout various researches, it has been confirmed that MUSIC is a fast, stable, and effective imaging technique. Furthermore, MUSIC can be extended in a straightforward fashion to the case of multiple non-overlapping inhomogeneities. Recently, by establishing relationships with Bessel functions of integer order, various intrinsic properties of MUSIC in both full- and limited-view and aperture inverse scattering problems have been revealed \cite{AJP,JKP,KP3,P-MUSIC1,P-MUSIC7,P-MUSIC8}.

In several studies, the MUSIC algorithm has been applied when one can use the complete elements of a multi-static response (MSR) matrix whose elements are measured scattered field or far-field pattern. However, under certain configurations, the diagonal elements of an MSR matrix cannot be handled. For example, it is very hard to simultaneously transmit and receive the signal in microwave imaging (see \cite{P-SUB11,BS,P-SUB9,P-MUSIC6,SLP,SSKLJ} for instance) so that the assumption that all elements of the MSR matrix are available cannot be used. This is the reason of the development of bistatic imaging technique to overcome intrinsic limitation of monostatic imaging, refer to \cite{C5,CKBP,LCJ,LSZM,QTWHW}. Fortunately, the shape of inhomogeneities can still be obtained via MUSIC without diagonal elements of MSR matrix. This fact can be examined through various numerical simulation; however, the theoretical reasons for its applicability have not been investigated. This provides a stimulus for analyzing the MUSIC algorithm without the diagonal elements of an MSR matrix.

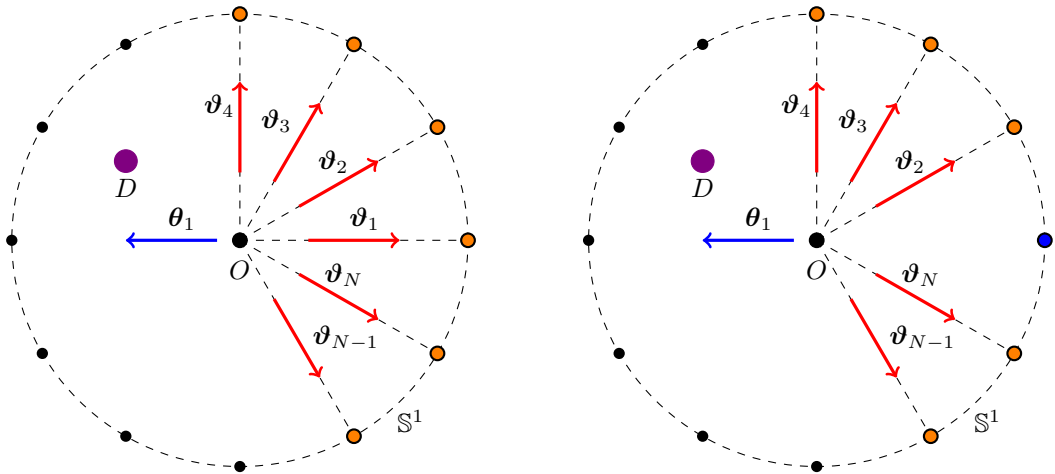
\begin{figure}[h]
\begin{center}
\begin{tikzpicture}[scale=3.0]
\draw[black,dashed] (1,0) arc (0:360:1);
\node at (0.75,-0.8) {$\mathbb{S}^1$};

\draw[black,dashed,-] (0,0) -- ({cos(0)},{sin(0)});
\draw[blue,very thick,solid,->] ({-0.1*cos(0)},{0.1*sin(0)}) -- node[above,black,xshift=4pt] {$\vt_1$} ({-0.5*cos(0)},{0.5*sin(0)});
\draw[red,very thick,solid,->] ({0.3*cos(0)},{0.3*sin(0)}) -- node[above,black,xshift=4pt] {$\vv_1$} ({0.7*cos(0)},{0.7*sin(0)});
\draw[black,thick,solid,fill=orange] ({cos(0)},{sin(0)}) circle (0.03cm);

\draw[black,dashed,-] (0,0) -- ({cos(30)},{sin(30)});
\draw[red,very thick,solid,->] ({0.3*cos(30)},{0.3*sin(30)}) -- node[above,black,xshift=-2pt] {$\vv_2$} ({0.7*cos(30)},{0.7*sin(30)});
\draw[black,thick,solid,fill=orange] ({cos(30)},{sin(30)}) circle (0.03cm);

\draw[black,dashed,-] (0,0) -- ({cos(60)},{sin(60)});
\draw[red,very thick,solid,->] ({0.3*cos(60)},{0.3*sin(60)}) -- node[above left,black,xshift=2pt] {$\vv_3$} ({0.7*cos(60)},{0.7*sin(60)});
\draw[black,thick,solid,fill=orange] ({cos(60)},{sin(60)}) circle (0.03cm);

\draw[black,dashed,-] (0,0) -- ({cos(90)},{sin(90)});
\draw[red,very thick,solid,->] ({0.3*cos(90)},{0.3*sin(90)}) -- node[above left,black,xshift=2pt] {$\vv_4$} ({0.7*cos(90)},{0.7*sin(90)});
\draw[black,thick,solid,fill=orange] ({cos(90)},{sin(90)}) circle (0.03cm);

\foreach \alpha in {120,150,...,270}
{\draw[black,thick,solid,fill=black] ({cos(\alpha)},{sin(\alpha)}) circle (0.02cm);}

\draw[black,dashed,-] (0,0) -- ({cos(300)},{sin(300)});
\draw[red,very thick,solid,->] ({0.3*cos(300)},{0.3*sin(300)}) -- node[right,black,xshift=2pt] {$\vv_{N-1}$} ({0.7*cos(300)},{0.7*sin(300)});
\draw[black,thick,solid,fill=orange] ({cos(300)},{sin(300)}) circle (0.03cm);

\draw[black,dashed,-] (0,0) -- ({cos(330)},{sin(330)});
\draw[red,very thick,solid,->] ({0.3*cos(330)},{0.3*sin(330)}) -- node[above,black,xshift=2pt] {$\vv_N$} ({0.7*cos(330)},{0.7*sin(330)});
\draw[black,thick,solid,fill=orange] ({cos(330)},{sin(330)}) circle (0.03cm);

\draw[black,thick,solid,fill=black] (0,0) circle (0.03cm) node[below,black,yshift=-4pt] {$O$};

\draw[violet,solid,fill=violet] (-0.5,0.35) circle (0.05cm) node[below,black,yshift=-3pt] {$D$};
\end{tikzpicture}\qquad\qquad
\begin{tikzpicture}[scale=3.0]
\draw[black,dashed] (1,0) arc (0:360:1);
\node at (0.75,-0.8) {$\mathbb{S}^1$};

\draw[blue,very thick,solid,->] ({-0.1*cos(0)},{0.1*sin(0)}) -- node[above,black,xshift=4pt] {$\vt_1$} ({-0.5*cos(0)},{0.5*sin(0)});

\draw[black,thick,solid,fill=blue] ({cos(0)},{sin(0)}) circle (0.03cm);

\draw[black,dashed,-] (0,0) -- ({cos(30)},{sin(30)});
\draw[red,very thick,solid,->] ({0.3*cos(30)},{0.3*sin(30)}) -- node[above,black,xshift=-2pt] {$\vv_2$} ({0.7*cos(30)},{0.7*sin(30)});
\draw[black,thick,solid,fill=orange] ({cos(30)},{sin(30)}) circle (0.03cm);

\draw[black,dashed,-] (0,0) -- ({cos(60)},{sin(60)});
\draw[red,very thick,solid,->] ({0.3*cos(60)},{0.3*sin(60)}) -- node[above left,black,xshift=2pt] {$\vv_3$} ({0.7*cos(60)},{0.7*sin(60)});
\draw[black,thick,solid,fill=orange] ({cos(60)},{sin(60)}) circle (0.03cm);

\draw[black,dashed,-] (0,0) -- ({cos(90)},{sin(90)});
\draw[red,very thick,solid,->] ({0.3*cos(90)},{0.3*sin(90)}) -- node[above left,black,xshift=2pt] {$\vv_4$} ({0.7*cos(90)},{0.7*sin(90)});
\draw[black,thick,solid,fill=orange] ({cos(90)},{sin(90)}) circle (0.03cm);

\foreach \alpha in {120,150,...,270}
{\draw[black,thick,solid,fill=black] ({cos(\alpha)},{sin(\alpha)}) circle (0.02cm);}

\draw[black,dashed,-] (0,0) -- ({cos(300)},{sin(300)});
\draw[red,very thick,solid,->] ({0.3*cos(300)},{0.3*sin(300)}) -- node[right,black,xshift=2pt] {$\vv_{N-1}$} ({0.7*cos(300)},{0.7*sin(300)});
\draw[black,thick,solid,fill=orange] ({cos(300)},{sin(300)}) circle (0.03cm);

\draw[black,dashed,-] (0,0) -- ({cos(330)},{sin(330)});
\draw[red,very thick,solid,->] ({0.3*cos(330)},{0.3*sin(330)}) -- node[above,black,xshift=2pt] {$\vv_N$} ({0.7*cos(330)},{0.7*sin(330)});
\draw[black,thick,solid,fill=orange] ({cos(330)},{sin(330)}) circle (0.03cm);

\draw[black,thick,solid,fill=black] (0,0) circle (0.03cm) node[below,black,yshift=-4pt] {$O$};

\draw[violet,solid,fill=violet] (-0.5,0.35) circle (0.05cm) node[below,black,yshift=-3pt] {$D$};

\end{tikzpicture}
\caption{\label{Configuration}Illustrations of traditional (left) and current (right) simulation configurations for the incident direction $\vt_1$.}
\end{center}
\end{figure}

In this study, we consider the MUSIC algorithm for imaging two-dimensional small inhomogeneities in TM and TE polarization from MSR matrix when the diagonal elements are cannot be handled. In order to show the feasibility, we carefully investigate the mathematical structure of a MUSIC-type imaging function by identifying a connection with the Bessel function of integer order of the first kind. This is based on the physical factorization of an MSR matrix in the presence of small inhomogeneities in TM and TE polarizations, refer to \cite{HSZ1}. The investigated structure explains why the location of inhomogeneities can be obtained via MUSIC in both TM and TE polarizations, and it reveals the undiscovered properties of MUSIC. In order to support the theoretical results, simulation results with synthetic data polluted by random noise are exhibited.

The paper is organized as follows. In Section \ref{sec:2}, we describe the two-dimensional direct scattering problem and introduce the far-field pattern in the presence of small inhomogeneities. In Section \ref{sec:3}, we introduce the traditional MUSIC algorithm. In Section \ref{sec:4}, we introduce the MUSIC algorithm, analyze the structure of the imaging function from the MSR matrix without diagonal elements, and discuss its properties. In Section \ref{sec:5}, we present the results of numerical simulations to support the analyzed structure of MUSIC. In Section \ref{sec:6}, we extend designed algorithm for imaging of small and extended perfectly conducting cracks and compare the imaging results according to the values of diagonal elements of MSR matrix. Finally, in Section \ref{sec:7}, we present a short conclusion including future work.

\section{Direct scattering problem and far-field pattern}\label{sec:2}
In this section, we introduce two-dimensional electromagnetic scattering in the presence of small inhomogeneity in TM and TE polarizations. For a detailed description, we recommend \cite{AK2,RC,VV} for a more detailed discussion. We assume that there exists a circle-like small inhomogeneity $D$ with radius $\alpha$ and center $\mz$, and every materials are characterized by the value of dielectric permittivity and magnetic permeability at the given angular frequency $\omega=2\pi f$. Here, $f$ denotes the ordinary frequency measured in \texttt{hertz}.

Let $\epsa$ and $\mua$ denote the value of dielectric permittivity and magnetic permeability of $D$, respectively, and we denote $\epsb$ and $\mub$ be those of $\mathbb{R}^2\backslash\overline{D}$. With this, the following piecewise constants of dielectric permittivity and magnetic permeability can be introduced;
\[\eps(\mx)=\left\{\begin{array}{ccl}
\epsa&\text{for}&\mx\in D\\
\epsb&\text{for}&\mx\in\mathbb{R}^2\backslash\overline{D}
\end{array}\right.
\quad\text{and}\quad
\mu(\mx)=\left\{\begin{array}{ccl}
\mua&\text{for}&\mx\in D\\
\mub&\text{for}&\mx\in\mathbb{R}^2\backslash\overline{D}.
\end{array}\right.\]
With this, we denote $\kb$ be the background wavenumber that satisfies $\kb^2=\omega^2\epsb\mub$.

In this paper, we consider the illumination of plane waves with the direction of propagation $\vt\in\mathbb{S}^1$:
\[u_{\inc}(\mx,\vt)=e^{i\kb \vt\cdot\mx},\]
where $\mathbb{S}^1$ is a two-dimensional unit circle centered at the origin. Then, the scattering of $u_{\inc}(\mx,\vt)$ by ${D}$ leads to the following direct scattering problem for the Helmholtz equation; let $u(\mx,\vt)$ be the time-harmonic total field; then, it satisfies
\begin{equation}\label{ForwardProblem}
\nabla\cdot\left(\frac{1}{\mu(\mx)}\nabla u(\mx,\vt)\right)+\omega^2\eps(\mx)u(\mx,\vt)=0\quad\text{for}\quad\mx\in\mathbb{R}^2\backslash\overline{{D}}
\end{equation}
with transmission conditions at the boundaries of ${D}$. We denote $u_{\scat}(\mx,\vt)=u(\mx,\vt)-u_{\inc}(\mx,\vt)$ as the scattered field, which is required to satisfy the Sommerfeld radiation condition
\[\lim_{|\mx|\to\infty}|\mx|^{\frac12}\left(\frac{\p u_{\scat}(\mx,\vt)}{\p|\mx|}-i\kb u_{\scat}(\mx,\vt)\right)=0\]
uniformly in all directions $\vv=\mx/\abs{\mx}\in\mathbb{S}^1$.
%

Let $u_\infty(\vv,\vt)$ be the far-field pattern of the scattered field $u_{\scat}(\mr,\vt)$ with observation direction $\vv\in\mathbb{S}^1$ that satisfies
\begin{equation}\label{FarFieldPattern}
u_{\scat}(\mr,\vt)=\frac{e^{i\kb |\mr|}}{\sqrt{|\mr|}}u_\infty(\vv,\vt)+O\left(\frac{1}{\sqrt{|\mr|}}\right)\quad\text{uniformly in all directions}\quad\vv=\frac{\mr}{|\mr|},\quad |\mr|\longrightarrow\infty.
\end{equation}
Then, by virtue of \cite{AK2}, the far-field pattern $u_\infty(\vv,\vt)$ can be represented as an asymptotic expansion formula, which plays a key role in designing the MUSIC algorithm.

\begin{lem}[Asymptotic Expansion Formula]
 For sufficiently large $\omega$, $u_\infty(\vv,\vt)$ can be represented as 
 \begin{equation}\label{AsymptoticFormula}
 u_\infty(\vv,\vt)=\alpha^2\pi\frac{\kb^2(1+i)}{4\sqrt{\kb\pi}}\left(\frac{\epsa-\epsb}{\sqrt{\epsb\mub}}-\frac{2\mub}{\mua+\mub}(\vv\cdot\vt))\right)e^{-i\kb(\vv-\vt)\cdot\mz}+o(\alpha^2).
 \end{equation}
\end{lem}
\section{Traditional MUSIC Algorithm}\label{sec:3}
In this section, we introduce the traditional MUSIC algorithm for imaging ${D}$ in dielectric permittivity (or TM polarization) and magnetic permeability (or TE polarization) cases. For the sake of simplicity, suppose that we have $N-$different number of incident and observation directions $\vt_n$ and $\vv_m$, respectively, for $n,m=1,2,\cdots,N$, and that the incident and observation directions are the same (i.e., $\vv_n=-\vt_n$). In this paper, we consider the full-view inverse problem. Therefore, we assume that $\vt_n$ is uniformly distributed in $\mathbb{S}^1$ such that
\begin{equation}\label{VectorTheta}
\vt_n=\left(\cos\frac{2\pi n}{N},\sin\frac{2\pi n}{N}\right).
\end{equation}

Traditionally, the following MSR matrix is used in the MUSIC algorithm:
\[\mathbb{M}=\begin{pmatrix}
u_{\infty}(\vv_1,\vt_1) & u_{\infty}(\vv_1,\vt_2) &  \cdots & u_{\infty}(\vv_1,\vt_{N-1}) & u_{\infty}(\vv_1,\vt_N)\\
u_{\infty}(\vv_2,\vt_1) & u_{\infty}(\vv_2,\vt_2) & \cdots & u_{\infty}(\vv_2,\vt_{N-1}) & u_{\infty}(\vv_2,\vt_N)\\
\vdots&\vdots&\ddots&\vdots&\vdots\\
u_{\infty}(\vv_N,\vt_1) & u_{\infty}(\vv_N,\vt_2) & \cdots & u_{\infty}(\vv_N,\vt_{N-1}) & u_{\infty}(\vv_N,\vt_N)\\
\end{pmatrix}.\]

First, let us assume that $\eps(\mx)\ne\epsb$ and $\mu(\mx)=\mub$. Based on \eqref{AsymptoticFormula}, since $u_\infty(\vv_m,\vt_n)$ can be approximated as
\[u_\infty(\vv_m,\vt_n)\approx\alpha^2\pi\frac{\kb^2(1+i)}{4\sqrt{\kb\pi}}\left(\frac{\epsa-\epsb}{\sqrt{\epsb\mub}}\right)e^{-i\kb(\vv_m-\vt_n)\cdot\mz},\]
$\mathbb{M}$ can be written as
\begin{equation}\label{MSR-TM}
\mathbb{M}=\alpha^2\pi\frac{\kb^2(1+i)}{4\sqrt{\kb\pi}}\left(\frac{\epsa-\epsb}{\sqrt{\epsb\mub}}\right)\mathbb{E}(\mz)^T\mathbb{E}(\mz),
\end{equation}
where
\begin{equation}\label{VecE-TM}
\mathbb{E}(\mz)=\bigg(e^{-i\kb\vv_1\cdot\mz},e^{-i\kb\vv_2\cdot\mz},\cdots,e^{-i\kb\vv_N\cdot\mz}\bigg)\bigg|_{\vv_m=-\vt_m} =\bigg(e^{i\kb\vt_1\cdot\mz},e^{i\kb\vt_2\cdot\mz},\cdots,e^{i\kb\vt_N\cdot\mz}\bigg).
\end{equation}
Based on the factorization of the MSR matrix, the range of $\mathbb{M}$ is determined by the span of $\mathbb{E}(\mz)$ corresponding to ${D}$; that is, we can define a signal subspace using a set of singular vectors corresponding to the nonzero singular values of $\mathbb{M}$.

Now, to introduce the imaging function of MUSIC, let us perform the singular value decomposition (SVD) of the MSR matrix $\mathbb{M}$:
\[\mathbb{M}=\sum_{n=1}^{N}\sigma_n\mE_n\mF_n^*\approx\sigma_1\mE_1\mF_1^*,\]
where superscript $*$ is the mark of Hermitian, $\mE_n$ and $\mF_n\in\mathbb{C}^{N\times 1}$ are respectively the left- and right-singular vectors of $\mathbb{M}$, and $\sigma_n$ denotes singular values that satisfy
\[\sigma_1>0\quad\text{and}\quad\sigma_n\approx0\quad\text{for}\quad n\geq 2.\]
Then, $\set{\mE_1}$ and $\set{\mE_{2},\mE_{3},\cdots,\mE_N}$ are the (orthonormal) basis of the signal and null (or noise) space of $\mathbb{M}$, respectively. Therefore, one can define the projection operator onto the noise subspace. This projection is given explicitly by
\[\mathbb{P}_{\noise}=\mathbb{I}_N-\mE_1\mE_1^*,\]
where $\mathbb{I}_N$ denotes the $N\times N$ identity matrix. By regarding the structure of $\mathbb{E}(\mz)$ of \eqref{VecE-TM}, we introduce the following unit test vector $\mf(\mx)$: for $\mx\in\Omega\subset\mathbb{R}^2$,
\begin{equation}\label{testvector-TM}
\mf(\mx)=\frac{1}{\sqrt{N}}\bigg(e^{i\kb\vt_1\cdot\mx},e^{i\kb\vt_2\cdot\mx},\ldots,e^{i\kb\vt_N\cdot\mx}\bigg)^T,
\end{equation}
where $\Omega$ denotes the region of interests (ROI). Then, there exists $N_0\in\mathbb{N}$ such that for any $N\geq N_0$, the following statement holds:
\[\mf(\mx)\in\text{Range}(\mathbb{M})\quad\text{if and only if}\quad\mx=\mz\in{D}.\]
This signifies that $\mathbb{P}_{\noise}(\mf(\mx))=0$ when $\mz\in{D}$. Thus, we can design a MUSIC-type imaging function such that
\begin{equation}\label{MUSICimaging-TM}
\mathfrak{F}_{\tm}(\mx)=\frac{1}{|\mathbb{P}_{\noise}(\mf(\mx))|}.
\end{equation}
Then, the map of $\mathfrak{F}_{\tm}(\mx)$ will have peaks of large ($+\infty$ in theory) and small amplitude at $\mx\in{D}$ and $\mx\in\Omega\backslash\overline{{D}}$, respectively.

Next, we assume that $\eps(\mx)=\epsb$ and $\mu(\mx)\ne\mub$. Based on \eqref{AsymptoticFormula}, since $u_\infty(\vv_m,\vt_n)$ can be approximated as
\begin{align*}
u_\infty(\vv_m,\vt_n)&=-\alpha^2\pi\frac{\kb^2(1+i)}{4\sqrt{\kb\pi}}\frac{2\mub}{\mua+\mub}(\vv\cdot\vt)e^{-i\kb(\vv-\vt)\cdot\mz}\\
&=-\alpha^2\pi\frac{\kb^2(1+i)}{4\sqrt{\kb\pi}}\frac{2\mub}{\mua+\mub}\left(\sum_{s=1}^{2}(\vv\cdot\me_s)(\vt\cdot\me_s)\right)e^{-i\kb(\vv-\vt)\cdot\mz},
\end{align*}
$\mathbb{M}$ can be written as
\begin{equation}\label{MSR-TE}
\mathbb{M}=\alpha^2\pi\frac{\kb^2(1+i)}{4\sqrt{\kb\pi}}\frac{2\mub}{\mua+\mub}\sum_{s=1}^{2}\mathbb{F}_s(\mz)^T\mathbb{F}_s(\mz),
\end{equation}
where
\begin{align}
\begin{aligned}\label{VecE-TE}
\mathbb{F}_s(\mz)&=\bigg((-\vv_1\cdot\me_s)e^{-i\kb\vv_1\cdot\mz},(-\vv_2\cdot\me_s)e^{-i\kb\vv_2\cdot\mz},\ldots,(-\vv_N\cdot\me_s)e^{-i\kb\vv_N\cdot\mz}\bigg)\bigg|_{\vv_m=-\vt_m}\\ &=\bigg((\vt_1\cdot\me_s)e^{i\kb\vt_1\cdot\mz},(\vt_2\cdot\me_s)e^{i\kb\vt_2\cdot\mz},\ldots,(\vt_N\cdot\me_s)e^{i\kb\vt_N\cdot\mz}\bigg).
\end{aligned}
\end{align}
Based on the factorization of the MSR matrix, the range of $\mathbb{M}$ is determined by the span of $\set{\mathbb{F}_1(\mz),\mathbb{F}_2(\mz)}$ corresponding to ${D}$. Thus, to introduce the imaging function of MUSIC, let us perform the singular value decomposition (SVD) of the MSR matrix $\mathbb{M}$:
\[\mathbb{M}=\sum_{n=1}^{N}\sigma_n\mE_n\mF_n^*\approx\sum_{n=1}^{2}\sigma_n\mE_n\mF_n^*.\]
Then, $\set{\mE_1,\mE_2}$ and $\set{\mE_{3},\mE_4,\ldots,\mE_N}$ are the (orthonormal) basis of the signal and null (or noise) space of $\mathbb{M}$, respectively. Therefore, one can define the projection operator onto the noise subspace. This projection is given explicitly by
\[\mathbb{P}_{\noise}=\mathbb{I}_N-\sum_{n=1}^{2}\mE_n\mE_n^*.\]
By regarding the structure of $\mathbb{F}_s(\mz)$ of \eqref{VecE-TE}, we introduce the following unit test vector $\mf(\mx)$: for $\mx\in\Omega$ and $\vx\in\mathbb{S}^1$,
\begin{equation}\label{testvector-TE}
\mf_\mu(\mx)=\sqrt{\frac{2}{N}}\bigg((\vt_1\cdot\vx)e^{i\kb\vt_1\cdot\mx},(\vt_2\cdot\vx)e^{i\kb\vt_2\cdot\mx},\ldots,(\vt_N\cdot\vx)e^{i\kb\vt_N\cdot\mx}\bigg)^T,
\end{equation}
where $\Omega$ denotes the region of interests (ROI). Then, there exists $N_0\in\mathbb{N}$ such that for any $N\geq N_0$, the following statement holds:
\[\mf_\mu(\mx)\in\text{Range}(\mathbb{M})\quad\text{if and only if}\quad\mx=\mz\in{D}.\]
This signifies that $\mathbb{P}_{\noise}(\mf(\mx))=0$ when $\mz\in{D}$. Thus, we can design a MUSIC-type imaging function such that
\[\mathfrak{F}_{\te}(\mx)=\frac{1}{|\mathbb{P}_{\noise}(\mf_{\mu}(\mx))|}.\]
Then, the map of $\mathfrak{F}_{\te}(\mx)$ will have peaks of large ($+\infty$ in theory) and small amplitude at $\mx\in{D}$ and $\mx\in\Omega\backslash\overline{{D}}$, respectively. A more detailed description is provided in \cite{AK2}.

\section{MUSIC algorithm without diagonal elements of MSR matrix: analysis and discussion}\label{sec:4}
Hereinafter, we assume that we have no information of $u_{\infty}(\vv_n,\vt_n)$ for $n=1,2,\cdots,N$. That is, the obtained MSR matrix must be of the following form:
\[\mathbb{K}=\begin{pmatrix}
\text{unknown} & u_{\infty}(\vv_1,\vt_2) &  \cdots & u_{\infty}(\vv_1,\vt_{N-1}) & u_{\infty}(\vv_1,\vt_N)\\
u_{\infty}(\vv_2,\vt_1) & \text{unknown} & \cdots & u_{\infty}(\vv_2,\vt_{N-1}) & u_{\infty}(\vv_2,\vt_N)\\
\vdots&\vdots&\ddots&\vdots&\vdots\\
u_{\infty}(\vv_N,\vt_1) & u_{\infty}(\vv_N,\vt_2) & \cdots & u_{\infty}(\vv_N,\vt_{N-1}) & \text{unknown}\\
\end{pmatrix}.\]
It should be noted that we have no any a priori information of the crack, and it is thus difficult to approximate the diagonal terms of the MSR matrix. Throughout this paper, we set the diagonal terms to be zero and consider the following MSR matrix:
\begin{equation}\label{MSR}
\mathbb{K}=\begin{pmatrix}
0 & u_{\infty}(\vv_1,\vt_2) &  \cdots & u_{\infty}(\vv_1,\vt_{N-1}) & u_{\infty}(\vv_1,\vt_N)\\
u_{\infty}(\vv_2,\vt_1) & 0 & \cdots & u_{\infty}(\vv_2,\vt_{N-1}) & u_{\infty}(\vv_2,\vt_N)\\
\vdots&\vdots&\ddots&\vdots&\vdots\\
u_{\infty}(\vv_N,\vt_1) & u_{\infty}(\vv_N,\vt_2) & \cdots & u_{\infty}(\vv_N,\vt_{N-1}) & 0\\
\end{pmatrix}.
\end{equation}
The remaining part of the algorithm is identical to the traditional one. For TM case, since the SVD of $\mathbb{K}$ can be written as
\begin{equation}\label{SVD}
  \mathbb{K}=\sum_{n=1}^{N}\sigma_n\mU_n\mV_n^*\approx\sigma_1\mU_1\mV_1^*,
  \end{equation}
we can define the projection operator onto the noise subspace
\begin{equation}\label{projection}
\mathbb{P}_{\noise}=\mathbb{I}_N-\mU_1\mU_1^*,
\end{equation}
and introduce the MUSIC-type imaging function $\mathfrak{F}_{\dm}$,\begin{equation}\label{MUSICimaging}
  \mathfrak{F}_{\dm}(\mz)=\frac{1}{|\mathbb{P}_{\noise}(\mf_\eps(\mx))|},
\end{equation}
where $\mf_{\eps}(\mx)$ is defined in \eqref{testvector-TM}. Then surprisingly, the location of $\mz\in{D}$ can be identified through the map of $\mathfrak{F}_{\dm}(\mx)$ when the total number of incident/observation directions $N$ is sufficiently large.

\begin{rem}
Although, the diagonal elements of the $\mathbb{K}$ are missing, total number of nonzero singular values is same as the total number of cracks but structure of singular values is quietly different from the ones of $\mathbb{M}$. This fact has been examined heuristically in \cite{P-SUB11}.
\end{rem}

For TE case, since the SVD of $\mathbb{K}$ can be written as
\begin{equation}\label{SVD}
  \mathbb{K}=\sum_{n=1}^{N}\sigma_n\mU_n\mV_n^*\approx\sum_{n=1}^{2}\sigma_n\mU_n\mV_n^*,
  \end{equation}
we can define the projection operator onto the noise subspace
\begin{equation}\label{projection}
\mathbb{P}_{\noise}=\mathbb{I}_N-\sum_{n=1}^{2}\mU_n\mU_n^*,
\end{equation}
and corresponding imaging function of the MUSIC can be introduced
\begin{equation}\label{MUSICimaging}
  \mathfrak{F}_{\de}^{(\eps)}(\mz)=\frac{1}{|\mathbb{P}_{\noise}(\mf_\eps(\mx))|}\quad\text{and}\quad\mathfrak{F}_{\de}^{(\mu)}(\mz)=\frac{1}{|\mathbb{P}_{\noise}(\mf_\mu(\mx))|},
\end{equation}
respectively. Here, $\mf_{\mu}(\mz)$ is defined in \eqref{testvector-TE}.

To confirm this applicability in both TM and TE cases, we establish mathematical structure of $\mathfrak{F}_{\dm}(\mx)$ by identifying a relationship with Bessel functions.

\subsection{Structure of the imaging function: TM case}

For proper analysis, we introduce a result derived in \cite{P-SUB3} that plays an important role in our analysis.
\begin{lem}\label{TheoremBessel-TM}
  For sufficiently large $N$, $\vt_n\in\mathbb{S}^1$ in \eqref{VectorTheta}, and $\mx\in\mathbb{R}^2$,
  \begin{equation}\label{Identity1}
  \frac{1}{N}\sum_{n=1}^{N}e^{i\kb\vt_n\cdot\mx}= J_0(\kb|\mx|),
  \end{equation}
where $J_n$ denotes the Bessel function of order $n$ of the first kind.
\end{lem}
Now, we can obtain the following results about the structure of $\mathfrak{F}_{\dm}(\mz)$.

\begin{thm}[TM polarization case]\label{Theorem-TM}
    For sufficiently large $N$ and $k$, $\mathfrak{F}_{\dm}(\mz)$ can be represented as follows:
    \begin{equation}\label{Structure-TM}
    \mathfrak{F}_{\dm}(\mx)=\left(\frac{N^2-2N+1}{N^2-2N}\right)\Big(1-J_0(\kb|\mx-\mz|)^2\Big)^{-1/2}.
    \end{equation}
\end{thm}
\begin{proof}
Based on \eqref{MSR-TM}, \eqref{MSR}, and \eqref{SVD}, $\mathbb{K}$ can be written
\[\mathbb{K}\approx\frac{\alpha^2\kb^2(1+i)(\epsa-\epsb)\pi}{4\sqrt{\kb\pi\epsb\mub}}\begin{pmatrix}
\medskip 0 & e^{i\kb(\vt_1+\vt_2)\cdot\mz} & \cdots & e^{i\kb(\vt_1+\vt_N)\cdot\mz}\\
e^{i\kb(\vt_2+\vt_1)\cdot\mz} & 0 & \cdots & e^{i\kb(\vt_2+\vt_N)\cdot\mz}\\
\medskip\vdots&\vdots&\ddots&\vdots\\
e^{i\kb(\vt_N+\vt_1)\cdot\mz} & e^{i\kb(\vt_N+\vt_2)\cdot\mz} & \cdots & 0\\
\end{pmatrix}.\]
Then, performing an elementary calculus yields
\[\frac{1}{(\sigma_1)^2}\mathbb{KK}^*=C_\eps
\begin{pmatrix}
N-1&(N-2)e^{i\kb(\vt_1-\vt_2)\cdot\mz}&\cdots&(N-2)e^{i\kb(\vt_1-\vt_N)\cdot\mz}\\
(N-2)e^{i\kb(\vt_2-\vt_1)\cdot\mz}&N-1&\cdots&(N-2)e^{i\kb(\vt_2-\vt_N)\cdot\mz}\\
\vdots&\vdots&\ddots&\vdots\\
(N-2)e^{i\kb(\vt_N-\vt_1)\cdot\mz}&(N-2)e^{i\kb(\vt_N-\vt_2)\cdot\mz}&\cdots&N-1
\end{pmatrix},\]
where
\[C_\eps=\left(\frac{\alpha^2\kb^2(\epsa-\epsb)\pi}{2\sigma_1\sqrt{\kb\pi\epsb\mub}}\right)^2.\]
Then, we have
\[\mathbb{I}-\mU_1\mU_1^*=\mathbb{I}-\frac{1}{(\sigma_1)^2}\mathbb{KK}^*=(1-C_\eps)\mathbb{I}-C_\eps(N-2)\begin{pmatrix}
e^{i\kb(\vt_1-\vt_1)\cdot\mz}&e^{i\kb(\vt_1-\vt_2)\cdot\mz}&\cdots&e^{i\kb(\vt_1-\vt_N)\cdot\mz}\\
e^{i\kb(\vt_2-\vt_1)\cdot\mz}&e^{i\kb(\vt_2-\vt_2)\cdot\mz}&\cdots&e^{i\kb(\vt_2-\vt_N)\cdot\mz}\\
\vdots&\vdots&\ddots&\vdots\\
e^{i\kb(\vt_N-\vt_1)\cdot\mz}&e^{i\kb(\vt_N-\vt_2)\cdot\mz}&\cdots&e^{i\kb(\vt_N-\vt_N)\cdot\mz}
\end{pmatrix}.\]

Based on \eqref{Identity1}, since
\begin{equation}\label{JacobiAnger}
e^{i\kb\vt_p\cdot\mz}\sum_{n=1}^{N}e^{i\kb\vt_n\cdot(\mx-\mz)}=Ne^{i\kb\vt_p\cdot\mz}J_0(\kb|\mx-\mz|)
\end{equation}
for $p=1,2,\cdots,N$, we have
\begin{align*}
\begin{pmatrix}
e^{i\kb(\vt_1-\vt_1)\cdot\mz}&e^{i\kb(\vt_1-\vt_2)\cdot\mz}&\cdots&e^{i\kb(\vt_1-\vt_N)\cdot\mz}\\
e^{i\kb(\vt_2-\vt_1)\cdot\mz}&e^{i\kb(\vt_2-\vt_2)\cdot\mz}&\cdots&e^{i\kb(\vt_2-\vt_N)\cdot\mz}\\
\vdots&\vdots&\ddots&\vdots\\
e^{i\kb(\vt_N-\vt_1)\cdot\mz}&e^{i\kb(\vt_N-\vt_2)\cdot\mz}&\cdots&e^{i\kb(\vt_N-\vt_N)\cdot\mz}
\end{pmatrix}\begin{pmatrix}
e^{i\kb\vt_1\cdot\mx}\\
e^{i\kb\vt_2\cdot\mx}\\
\vdots\\
e^{i\kb\vt_N\cdot\mx}\\
\end{pmatrix}
=\begin{pmatrix}
\displaystyle Ne^{i\kb\vt_1\cdot\mz}J_0(\kb|\mx-\mz|)\\
\displaystyle Ne^{i\kb\vt_2\cdot\mz}J_0(\kb|\mx-\mz|)\\
\vdots\\
\displaystyle Ne^{i\kb\vt_N\cdot\mz}J_0(\kb|\mx-\mz|)
\end{pmatrix}
\end{align*}
and correspondingly,
\[\mathbb{P}_{\noise}(\mf(\mx))=(\mathbb{I}-\mU_1\mU_1^*)\mf(\mx)=\frac{1-C_\eps}{\sqrt{N}}\begin{pmatrix}
\medskip e^{i\kb\vt_1\cdot\mx}\\
e^{i\kb\vt_2\cdot\mx}\\
\medskip \vdots\\
e^{i\kb\vt_N\cdot\mx}\\
\end{pmatrix}
-C_\eps(N-2)\sqrt{N}\begin{pmatrix}
\medskip  e^{i\kb\vt_1\cdot\mz}J_0(\kb|\mx-\mz|)\\
 e^{i\kb\vt_2\cdot\mz}J_0(\kb|\mx-\mz|)\\
\medskip \vdots\\
 e^{i\kb\vt_N\cdot\mz}J_0(\kb|\mx-\mz|)
\end{pmatrix}.\]

Now, let us write
\[|\mathbb{P}_{\noise}(\mf(\mx))|=\left(\mathbb{P}_{\noise}(\mf(\mx))\cdot\overline{\mathbb{P}_{\noise}(\mf(\mx))}\right)^{1/2}=\left(\sum_{n=1}^{N}\bigg(\frac{(1-C_\eps)^2}{N}-(\Psi_1+\overline{\Psi}_1)+\Psi_2\overline\Psi_2\bigg)\right)^{1/2},\]
where
\begin{align*}
\Psi_1&=(1-C_\eps)C_\eps(N-2)e^{i\kb\vt_n\cdot(\mx-\mz)}J_0(\kb|\mx-\mz|)\\
\Psi_2&=C_\eps(N-2)\sqrt{N}e^{i\kb\vt_n\cdot\mz}J_0(\kb|\mx-\mz|).
\end{align*}
Applying \eqref{JacobiAnger} again, we can obtain
\[\sum_{n=1}^{N}(\Psi_1+\overline{\Psi}_1)=2C_\eps(1-C_\eps)(N-2)NJ_0(\kb|\mx-\mz|)^2\]
and
\[\sum_{n=1}^{N}\Psi_2\overline\Psi_2=C_\eps^2(N-2)^2N^2J_0(\kb|\mx-\mz|)^2,\]
$|\mathbb{P}_{\noise}(\mf_\eps(\mx))|$ becomes
\[|\mathbb{P}_{\noise}(\mf_\eps(\mx))|=\Big((1-C_\eps)^2-2(1-C_\eps)C_\eps(N-2)NJ_0(\kb|\mx-\mz|)^2+C_\eps^2(N-2)^2N^2J_0(\kb|\mx-\mz|)^2\Big)^{1/2}.\]
Since $|\mathbb{P}_{\noise}(\mf_\eps(\mz))|=0$ and $J_0(0)=1$,
\[(1-C_\eps)^2-2(1-C_\eps)C_\eps(N-2)N+C_\eps^2(N-2)^2N^2=\Big((1-C_\eps)-C_\eps N(N-2)\Big)^2=0\quad\text{implies}\quad C_\eps=\frac{1}{(N-1)^2}.\]
Therefore,
\[|\mathbb{P}_{\noise}(\mf(\mx))|=(1-C_\eps)\Big(1-J_0(\kb|\mx-\mz|)^2\Big)^{1/2}=\left(\frac{N^2-2N}{N^2-2N+1}\right)\Big(1-J_0(\kb|\mx-\mz|)^2\Big)^{1/2}.\]
Hence, we can derive the \eqref{Structure-TM}.
\end{proof}

\begin{rem}\label{Remark-TM}Based on the identified structure \eqref{Structure-TM}, we can observe several properties of MUSIC.
\begin{enumerate}
\item Since $J_0(0)=1$, the map of $\mathfrak{F}_{\dm}(\mx)$ will contain peak of large magnitude (theoretically $+\infty$) at $\mx=\mz\in{D}$. This explains why MUSIC is applicable for imaging or identifying cracks without diagonal elements of an MSR matrix.
\item As in traditional MUSIC-type imaging, the resolution of the imaging result is highly dependent on the values of $k$ and $N$. This signifies that because $\mathfrak{F}_{\dm}(\mx)$ is related to $J_0$, a result of poor resolution will obtained if $k$ is small. In contrast, if $k$ is sufficiently large (as we assume in Theorem \ref{Theorem-TM}), a result of high resolution can be obtained even in the presence of various artifacts. 
\item Based on recent work \cite[Section 3.5]{P-MUSIC1}, $f_{\fm}(\mx)$ can be represented as
\[f_{\fm}(\mx)=\Big(1-J_0(\kb|\mx-\mz|)^2\Big)^{-1/2}.\]
Since
\[\lim_{N\to\infty}\mathfrak{F}_{\dm}(\mx)=\lim_{N\to\infty}\left(\frac{N^2-2N+1}{N^2-2N}\right)\Big(1-J_0(\kb|\mx-\mz|)^2\Big)^{-1/2}=\Big(1-J_0(\kb|\mx-\mz|)^2\Big)^{-1/2}=f_{\fm}(\mx),\]
we can examine that several incident/observation directions should be applied to increase the resolution of the imaging results, i.e, a result of poor resolution will appear if $N$ is small.
\end{enumerate}
\end{rem}

Based on the Remark \ref{Remark-TM}, we can also examine the unique determination.
\begin{cor}[Unique determination of inhomogeneity]\label{CorUnique}
For sufficiently large $N$ and $k$, the location of small inhomogeneity can be identified uniquely via the map of $\mathfrak{F}_{\dm}(\mx)$.
\end{cor}
%
%

\subsection{Structure of the imaging function: TE case}
Now, we analyze the imaging function in TE polarization. We first recall a useful result derived in \cite{P-SUB3,P-SUB8}.

\begin{lem}\label{TheoremBessel-TE}
  For sufficiently large $N$, $\vt_n\in\mathbb{S}^1$ in \eqref{VectorTheta}, $\vx\in\mathbb{S}^1$, and $\mz\in\mathbb{R}^2$,
  \begin{align}
  \begin{aligned}\label{Identity2}
  &\sum_{n=1}^{N}(\vx\cdot\vt_n)e^{i\kb\vt_n\cdot\mz}=iN\left(\frac{\mz}{|\mz|}\cdot\vx\right)J_1(\kb|\mz|),\\
  &\sum_{n=1}^{N}(\vx\cdot\vt_n)^2=\frac{N}{2},\\
  &\sum_{n=1}^{N}(\vt_m\cdot\vt_n)(\vx\cdot\vt_n)e^{i\kb\vt_n\cdot\mz}=\frac{N}{2}(\vt_m\cdot\vx)\Big(J_0(\kb|\mz|)+J_2(\kb|\mz|)\Big)-N\left(\frac{\mz}{|\mz|}\cdot\vt_m\right)\left(\frac{\mz}{|\mz|}\cdot\vx\right)J_2(\kb|\mz|).
  \end{aligned}
  \end{align}
\end{lem}


\begin{thm}[TE polarization case]\label{Theorem-TE}
    Let $\vx=(\cos\xi,\sin\xi)$. Then, for sufficiently large $N$ and $k$, $\mathfrak{F}_{\de}^{(\eps)}(\mx)$ and $\mathfrak{F}_{\de}^{(\mu)}(\mx)$ can be represented as follows:
    \begin{equation}\label{Structure-TE1}
    \mathfrak{F}_{\de}^{(\eps)}(\mx)=\left((1-C_\mu)^2-2C_\mu(1-C_\mu)N\left(\frac{N}{2}-2\right)J_1(\kb|\mx-\mz|)^2+C_\mu^2\frac{N^2}{2}\left(\frac{N}{2}-2\right)^2J_1(\kb|\mx-\mz|)^2\right)^{-1/2}
    \end{equation}
    and
    \begin{equation}\label{Structure-TE2}
    \mathfrak{F}_{\de}^{(\mu)}(\mx)=\left((1-C_\mu)^2-2C_\mu(1-C_\mu)N\left(\frac{N}{2}-2\right)J_1(\kb|\mx-\mz|)^2+C_\mu^2\frac{N^2}{2}\left(\frac{N}{2}-2\right)^2J_1(\kb|\mx-\mz|)^2\right)^{-1/2},
    \end{equation}
    where $\sigma_1\approx\sigma_2=\sigma$ and 
    \[C_\mu=\left(\frac{\alpha^2\kb^2\mub\pi}{\sigma(\mua+\mub)\sqrt{\kb\pi}}\right)^2.\]
\end{thm}
\begin{proof}
Based on \eqref{MSR-TE}, \eqref{MSR}, and \eqref{SVD}, $\mathbb{K}$ can be written
\[\mathbb{K}\approx\frac{\alpha^2\kb^2(1+i)\mub\pi}{2(\mua+\mub)\sqrt{\kb\pi}}
\begin{pmatrix}
\medskip 0 & (\vt_1\cdot\vt_2)e^{i\kb(\vt_1+\vt_2)\cdot\mz} & \cdots & (\vt_1\cdot\vt_N)e^{i\kb(\vt_1+\vt_N)\cdot\mz}\\
(\vt_2\cdot\vt_1)e^{i\kb(\vt_2+\vt_1)\cdot\mz} & 0 & \cdots & (\vt_2\cdot\vt_N)e^{i\kb(\vt_2+\vt_N)\cdot\mz}\\
\medskip\vdots&\vdots&\ddots&\vdots\\
(\vt_N\cdot\vt_1)e^{i\kb(\vt_N+\vt_1)\cdot\mz} & (\vt_N\cdot\vt_2)e^{i\kb(\vt_N+\vt_2)\cdot\mz} & \cdots & 0\\
\end{pmatrix}.\]
Then, since $\sigma_1\approx\sigma_2=\sigma$, performing an elementary calculus yields
\begin{align*}
&\frac{1}{\sigma^2}\mathbb{KK}^*\\
&=C_\mu
\begin{pmatrix}
N/2-1&(N/2-2)(\vt_1\cdot\vt_2)e^{i\kb(\vt_1-\vt_2)\cdot\mz}&\cdots&(N/2-2)(\vt_1\cdot\vt_N)e^{i\kb(\vt_1-\vt_N)\cdot\mz}\\
(N/2-2)(\vt_2\cdot\vt_1)e^{i\kb(\vt_2-\vt_1)\cdot\mz}&N/2-1&\cdots&(N/2-2)(\vt_2\cdot\vt_N)e^{i\kb(\vt_2-\vt_N)\cdot\mz}\\
\medskip\vdots&\vdots&\ddots&\vdots\\
(N/2-2)(\vt_N\cdot\vt_1)e^{i\kb(\vt_N-\vt_1)\cdot\mz}&
(N/2-2)(\vt_N\cdot\vt_2)e^{i\kb(\vt_N-\vt_2)\cdot\mz}&\cdots&N/2-1
\end{pmatrix}
\end{align*}
and correspondingly,
\begin{multline*}
\mathbb{I}-\sum_{m=1}^{2}\mU_m\mU_m^*=\mathbb{I}-\frac{1}{\sigma^2}\mathbb{KK}^*=\\
(1-C_\mu)\mathbb{I}-C_\mu\left(\frac{N}{2}-2\right)\begin{pmatrix}
(\vt_1\cdot\vt_1)e^{i\kb(\vt_1-\vt_1)\cdot\mz}&(\vt_1\cdot\vt_2)e^{i\kb(\vt_1-\vt_2)\cdot\mz}&\cdots&(\vt_1\cdot\vt_N)e^{i\kb(\vt_1-\vt_N)\cdot\mz}\\
(\vt_2\cdot\vt_1)e^{i\kb(\vt_2-\vt_1)\cdot\mz}&(\vt_2\cdot\vt_2)e^{i\kb(\vt_2-\vt_2)\cdot\mz}&\cdots&(\vt_2\cdot\vt_N)e^{i\kb(\vt_2-\vt_N)\cdot\mz}\\
\vdots&\vdots&\ddots&\vdots\\
(\vt_N\cdot\vt_1)e^{i\kb(\vt_N-\vt_1)\cdot\mz}&(\vt_N\cdot\vt_2)e^{i\kb(\vt_N-\vt_2)\cdot\mz}&\cdots&(\vt_N\cdot\vt_N)e^{i\kb(\vt_N-\vt_N)\cdot\mz}
\end{pmatrix}.
\end{multline*}

First, we consider the structure of $\mathbb{P}_{\noise}(\mf_\eps(\mx))$. Based on \eqref{Identity2}, since
\begin{align*}
&\begin{pmatrix}
(\vt_1\cdot\vt_1)e^{i\kb(\vt_1-\vt_1)\cdot\mz}&(\vt_1\cdot\vt_2)e^{i\kb(\vt_1-\vt_2)\cdot\mz}&\cdots&(\vt_1\cdot\vt_N)e^{i\kb(\vt_1-\vt_N)\cdot\mz}\\
(\vt_2\cdot\vt_1)e^{i\kb(\vt_2-\vt_1)\cdot\mz}&(\vt_2\cdot\vt_2)e^{i\kb(\vt_2-\vt_2)\cdot\mz}&\cdots&(\vt_2\cdot\vt_N)e^{i\kb(\vt_2-\vt_N)\cdot\mz}\\
\vdots&\vdots&\ddots&\vdots\\
(\vt_N\cdot\vt_1)e^{i\kb(\vt_N-\vt_1)\cdot\mz}&(\vt_N\cdot\vt_2)e^{i\kb(\vt_N-\vt_2)\cdot\mz}&\cdots&(\vt_N\cdot\vt_N)e^{i\kb(\vt_N-\vt_N)\cdot\mz}
\end{pmatrix}\begin{pmatrix}
e^{i\kb\vt_1\cdot\mx}\\
e^{i\kb\vt_2\cdot\mx}\\
\vdots\\
e^{i\kb\vt_N\cdot\mx}\\
\end{pmatrix}\\
&=\begin{pmatrix}
\displaystyle\medskip e^{i\kb\vt_1\cdot\mz}\sum_{n=1}^{N}(\vt_1\cdot\vt_n)e^{i\kb\vt_n\cdot(\mx-\mz)}\\
\displaystyle e^{i\kb\vt_2\cdot\mz}\sum_{n=1}^{N}(\vt_2\cdot\vt_n)e^{i\kb\vt_n\cdot(\mx-\mz)}\\
\vdots\\
\displaystyle e^{i\kb\vt_N\cdot\mz}\sum_{n=1}^{N}(\vt_N\cdot\vt_n)e^{i\kb\vt_n\cdot(\mx-\mz)}
\end{pmatrix}
=\begin{pmatrix}
\displaystyle\medskip iNe^{i\kb\vt_1\cdot\mz}\left(\frac{\mx-\mz}{|\mx-\mz|}\cdot\vt_1\right)J_1(\kb|\mx-\mz|)\\
\displaystyle iNe^{i\kb\vt_2\cdot\mz}\left(\frac{\mx-\mz}{|\mx-\mz|}\cdot\vt_2\right)J_1(\kb|\mx-\mz|)\\
\vdots\\
\displaystyle iNe^{i\kb\vt_N\cdot\mz}\left(\frac{\mx-\mz}{|\mx-\mz|}\cdot\vt_N\right)J_1(\kb|\mx-\mz|)
\end{pmatrix},
\end{align*}
we can examine that
\begin{align*}
\mathbb{P}_{\noise}(\mf_\eps(\mx))&=\left(\mathbb{I}-\sum_{m=1}^{2}\mU_m\mU_m^*\right)\mf_\eps(\mx)\\
&=\frac{1-C_\mu}{\sqrt{N}}\begin{pmatrix}
e^{i\kb\vt_1\cdot\mx}\\
e^{i\kb\vt_2\cdot\mx}\\
\vdots\\
e^{i\kb\vt_N\cdot\mx}\\
\end{pmatrix}
-iC_\mu\left(\frac{N}{2}-2\right)\sqrt{N}
\begin{pmatrix}
\displaystyle\medskip e^{i\kb\vt_1\cdot\mz}\left(\frac{\mx-\mz}{|\mx-\mz|}\cdot\vt_1\right)J_1(\kb|\mx-\mz|)\\
\displaystyle e^{i\kb\vt_2\cdot\mz}\left(\frac{\mx-\mz}{|\mx-\mz|}\cdot\vt_2\right)J_1(\kb|\mx-\mz|)\\
\vdots\\
\displaystyle e^{i\kb\vt_N\cdot\mz}\left(\frac{\mx-\mz}{|\mx-\mz|}\cdot\vt_N\right)J_1(\kb|\mx-\mz|)
\end{pmatrix}.
\end{align*}

Now, let us write
\[|\mathbb{P}_{\noise}(\mf_\eps(\mx))|=\left(\mathbb{P}_{\noise}(\mf_\eps(\mx))\cdot\overline{\mathbb{P}_{\noise}(\mf_\eps(\mx))}\right)^{1/2}=\left(\sum_{n=1}^{N}\bigg(\frac{(1-C_\mu)^2}{N}+(\Psi_3+\overline{\Psi}_3)+\Psi_4\overline\Psi_4\bigg)\right)^{1/2},\]
where
\begin{align*}
\Psi_3&=i(1-C_\mu)C_\mu\left(\frac{N}{2}-2\right)e^{i\kb\vt_n\cdot(\mx-\mz)}\left(\frac{\mx-\mz}{|\mx-\mz|}\cdot\vt_n\right)J_1(\kb|\mx-\mz|)\\
\Psi_4&=-iC_\mu\left(\frac{N}{2}-2\right)\sqrt{N}e^{i\kb\vt_n\cdot\mz}\left(\frac{\mx-\mz}{|\mx-\mz|}\cdot\vt_n\right)J_1(\kb|\mx-\mz|).
\end{align*}

Applying \eqref{Identity2}, we can examine that
\[\sum_{n=1}^{N}\left(\frac{\mx-\mz}{|\mx-\mz|}\cdot\vt_n\right)e^{i\kb\vt_n\cdot(\mx-\mz)}=iN\left(\frac{\mx-\mz}{|\mx-\mz|}\cdot\frac{\mx-\mz}{|\mx-\mz|}\right)J_1(\kb|\mx-\mz|)=J_1(\kb|\mx-\mz|)\]
Hence,
\[\sum_{n=1}^{N}(\Psi_3+\overline{\Psi}_3)=-2N(1-C_\mu)C_\mu\left(\frac{N}{2}-2\right)J_1(\kb|\mx-\mz|)^2.\]
Now, applying \eqref{Identity2} again, we can examine that
\[\sum_{n=1}^{N}\Psi_4\overline\Psi_4=C_\mu^2\left(\frac{N}{2}-2\right)^2NJ_1(\kb|\mx-\mz|)^2\sum_{n=1}^{N}\left(\frac{\mx-\mz}{|\mx-\mz|}\cdot\vt_n\right)^2=\frac{N^2}{2}C_\mu^2\left(\frac{N}{2}-2\right)^2J_1(\kb|\mx-\mz|)^2.\]
Therefore,
\[|\mathbb{P}_{\noise}(\mf_\eps(\mx))|=\left((1-C_\mu)^2-2N(1-C_\mu)C_\mu\left(\frac{N}{2}-2\right)J_1(\kb|\mx-\mz|)^2+\frac{N^2}{2}C_\mu^2\left(\frac{N}{2}-2\right)^2J_1(\kb|\mx-\mz|)^2\right)^{1/2}.\]

Next, we consider the structure of $\mathbb{P}_{\noise}(\mf_\mu(\mx))$. Based on \eqref{Identity2}, since
\begin{align*}
&\begin{pmatrix}
(\vt_1\cdot\vt_1)e^{i\kb(\vt_1-\vt_1)\cdot\mz}&(\vt_1\cdot\vt_2)e^{i\kb(\vt_1-\vt_2)\cdot\mz}&\cdots&(\vt_1\cdot\vt_N)e^{i\kb(\vt_1-\vt_N)\cdot\mz}\\
(\vt_2\cdot\vt_1)e^{i\kb(\vt_2-\vt_1)\cdot\mz}&(\vt_2\cdot\vt_2)e^{i\kb(\vt_2-\vt_2)\cdot\mz}&\cdots&(\vt_2\cdot\vt_N)e^{i\kb(\vt_2-\vt_N)\cdot\mz}\\
\vdots&\vdots&\ddots&\vdots\\
(\vt_N\cdot\vt_1)e^{i\kb(\vt_N-\vt_1)\cdot\mz}&(\vt_N\cdot\vt_2)e^{i\kb(\vt_N-\vt_2)\cdot\mz}&\cdots&(\vt_N\cdot\vt_N)e^{i\kb(\vt_N-\vt_N)\cdot\mz}
\end{pmatrix}\begin{pmatrix}
(\vt_1\cdot\vx)e^{i\kb\vt_1\cdot\mx}\\
(\vt_2\cdot\vx)e^{i\kb\vt_2\cdot\mx}\\
\vdots\\
(\vt_N\cdot\vx)e^{i\kb\vt_N\cdot\mx}\\
\end{pmatrix}\\
&=\begin{pmatrix}
\displaystyle\medskip e^{i\kb\vt_1\cdot\mz}\sum_{n=1}^{N}(\vt_1\cdot\vt_n)(\vx\cdot\vt_n)e^{i\kb\vt_n\cdot(\mx-\mz)}\\
\displaystyle e^{i\kb\vt_2\cdot\mz}\sum_{n=1}^{N}(\vt_2\cdot\vt_n)(\vx\cdot\vt_n)e^{i\kb\vt_n\cdot(\mx-\mz)}\\
\vdots\\
\displaystyle e^{i\kb\vt_N\cdot\mz}\sum_{n=1}^{N}(\vt_N\cdot\vt_n)(\vx\cdot\vt_n)e^{i\kb\vt_n\cdot(\mx-\mz)}
\end{pmatrix}\\
&=\begin{pmatrix}
\displaystyle\medskip \frac{N}{2}(\vt_1\cdot\vx)e^{i\kb\vt_1\cdot\mz}\Big(J_0(\kb|\mx-\mz|)+J_2(\kb|\mx-\mz|)\Big)-N\left(\frac{\mx-\mz}{|\mx-\mz|}\cdot\vt_1\right)\left(\frac{\mx-\mz}{|\mx-\mz|}\cdot\vx\right)J_2(\kb|\mx-\mz|)\\
\displaystyle \frac{N}{2}(\vt_2\cdot\vx)e^{i\kb\vt_2\cdot\mz}\Big(J_0(\kb|\mx-\mz|)+J_2(\kb|\mx-\mz|)\Big)-N\left(\frac{\mx-\mz}{|\mx-\mz|}\cdot\vt_2\right)\left(\frac{\mx-\mz}{|\mx-\mz|}\cdot\vx\right)J_2(\kb|\mx-\mz|)\\
\vdots\\
\displaystyle \frac{N}{2}(\vt_N\cdot\vx)e^{i\kb\vt_N\cdot\mz}\Big(J_0(\kb|\mx-\mz|)+J_2(\kb|\mx-\mz|)\Big)-N\left(\frac{\mx-\mz}{|\mx-\mz|}\cdot\vt_N\right)\left(\frac{\mx-\mz}{|\mx-\mz|}\cdot\vx\right)J_2(\kb|\mx-\mz|)
\end{pmatrix},
\end{align*}
we can examine that
\begin{multline*}
\mathbb{P}_{\noise}(\mf_\mu(\mx))=\left(\mathbb{I}-\sum_{m=1}^{2}\mU_m\mU_m^*\right)\mf_\mu(\mx)=(1-C_\mu)\sqrt{\frac{2}{N}}\begin{pmatrix}
(\vt_1\cdot\vx)e^{i\kb\vt_1\cdot\mx}\\
(\vt_2\cdot\vx)e^{i\kb\vt_2\cdot\mx}\\
\vdots\\
(\vt_N\cdot\vx)e^{i\kb\vt_N\cdot\mx}\\
\end{pmatrix}\\
-C_\mu\left(\frac{N}{2}-2\right)\sqrt{2N}
\begin{pmatrix}
\displaystyle\medskip \frac{(\vt_1\cdot\vx)}{2}e^{i\kb\vt_1\cdot\mz}\Big(J_0(\kb|\mx-\mz|)+J_2(\kb|\mx-\mz|)\Big)-\left(\frac{\mx-\mz}{|\mx-\mz|}\cdot\vt_1\right)\left(\frac{\mx-\mz}{|\mx-\mz|}\cdot\vx\right)J_2(\kb|\mx-\mz|)\\
\displaystyle \frac{(\vt_2\cdot\vx)}{2}e^{i\kb\vt_2\cdot\mz}\Big(J_0(\kb|\mx-\mz|)+J_2(\kb|\mx-\mz|)\Big)-\left(\frac{\mx-\mz}{|\mx-\mz|}\cdot\vt_2\right)\left(\frac{\mx-\mz}{|\mx-\mz|}\cdot\vx\right)J_2(\kb|\mx-\mz|)\\
\vdots\\
\displaystyle \frac{(\vt_N\cdot\vx)}{2}e^{i\kb\vt_N\cdot\mz}\Big(J_0(\kb|\mx-\mz|)+J_2(\kb|\mx-\mz|)\Big)-\left(\frac{\mx-\mz}{|\mx-\mz|}\cdot\vt_N\right)\left(\frac{\mx-\mz}{|\mx-\mz|}\cdot\vx\right)J_2(\kb|\mx-\mz|)
\end{pmatrix}.
\end{multline*}

Now, let us write
\[|\mathbb{P}_{\noise}(\mf_\mu(\mx))|=\left(\mathbb{P}_{\noise}(\mf_\mu(\mx))\cdot\overline{\mathbb{P}_{\noise}(\mf_\mu(\mx))}\right)^{1/2}=\left(\sum_{n=1}^{N}\bigg(\frac{2(1-C_\mu)^2}{N}+(\Psi_5+\overline{\Psi}_5)+\Psi_6\overline\Psi_6\bigg)\right)^{1/2},\]
where
\begin{align*}
\Psi_5=&(1-C_\mu)C_\mu\left(\frac{N}{2}-2\right)(\vt_n\cdot\vx)^2e^{i\kb\vt_1\cdot(\mx-\mz)}\Big(J_0(\kb|\mx-\mz|)+J_2(\kb|\mx-\mz|)\Big)\\
&-2(1-C_\mu)C_\mu\left(\frac{N}{2}-2\right)(\vt_n\cdot\vx)\left(\frac{\mx-\mz}{|\mx-\mz|}\cdot\vt_n\right)\left(\frac{\mx-\mz}{|\mx-\mz|}\cdot\vx\right)J_2(\kb|\mx-\mz|)\\
\Psi_6=&C_\mu\left(\frac{N}{2}-2\right)\sqrt{2N}\left[\frac{(\vt_n\cdot\vx)}{2}e^{i\kb\vt_n\cdot\mz}\Big(J_0(\kb|\mx-\mz|)+J_2(\kb|\mx-\mz|)\Big)\right.\\
&\left.-\left(\frac{\mx-\mz}{|\mx-\mz|}\cdot\vt_n\right)\left(\frac{\mx-\mz}{|\mx-\mz|}\cdot\vx\right)J_2(\kb|\mx-\mz|)\right].
\end{align*}
Now, applying \eqref{Identity2}, we can evaluate
\begin{align*}
\sum_{n=1}^{N}(\Psi_5+\overline{\Psi}_5)&=(1-C_\mu)C_\mu\left(\frac{N}{2}-2\right)\Big(J_0(\kb|\mx-\mz|)+J_2(\kb|\mx-\mz|)\Big)\left(\sum_{n=1}^{N}(\vt_n\cdot\vx)^2\big(e^{i\kb\vt_1\cdot(\mx-\mz)}+e^{-i\kb\vt_1\cdot(\mx-\mz)}\big)\right)\\
&=N(1-C_\mu)C_\mu\left(\frac{N}{2}-2\right)\Big(J_0(\kb|\mx-\mz|)+J_2(\kb|\mx-\mz|)\Big)^2
\end{align*}
and
\begin{align*}
&\sum_{n=1}^{N}(\Psi_6\overline{\Psi}_6)=2NC_\mu^2\left(\frac{N}{2}-2\right)^2\left[\sum_{n=1}^{N}\frac{(\vt_n\cdot\vx)^2}{4}\Big(J_0(\kb|\mx-\mz|)+J_2(\kb|\mx-\mz|)\Big)^2\right.\\
&.-\frac{1}{2}\left(\frac{\mx-\mz}{|\mx-\mz|}\cdot\vx\right)\Big(J_0(\kb|\mx-\mz|)+J_2(\kb|\mx-\mz|)\Big)J_2(\kb|\mx-\mz|)\left\{\sum_{n=1}^{N}(\vt_n\cdot\vx)\left(\frac{\mx-\mz}{|\mx-\mz|}\cdot\vt_n\right)\big(e^{i\kb\vt_n\cdot\mz}+e^{-i\kb\vt_n\cdot\mz}\big)\right\}\\
&\left.+\sum_{n=1}^{N}\left(\frac{\mx-\mz}{|\mx-\mz|}\cdot\vt_n\right)^2\left(\frac{\mx-\mz}{|\mx-\mz|}\cdot\vx\right)^2J_2(\kb|\mx-\mz|)^2\right]\\
&=2NC_\mu^2\left(\frac{N}{2}-2\right)^2\left[\frac{N}{8}\Big(J_0(\kb|\mx-\mz|)+J_2(\kb|\mx-\mz|)\Big)^2\right.\\
&.-\frac{1}{2}\left(\frac{\mx-\mz}{|\mx-\mz|}\cdot\vx\right)\Big(J_0(\kb|\mx-\mz|)+J_2(\kb|\mx-\mz|)\Big)J_2(\kb|\mx-\mz|)\\
&\times\left\{N\left(\frac{\mx-\mz}{|\mx-\mz|}\cdot\vx\right)\Big(J_0(\kb|\mx-\mz|)+J_2(\kb|\mx-\mz|)\Big)\left(\frac{\mx-\mz}{|\mx-\mz|}\cdot\vt_n\right)\left(\frac{\mx-\mz}{|\mx-\mz|}\cdot\vx\right)J_2(\kb|\mx-\mz|)\right\}\\
&\left.+\left(\frac{\mx-\mz}{|\mx-\mz|}\cdot\vt_n\right)^2\left(\frac{\mx-\mz}{|\mx-\mz|}\cdot\vx\right)^2J_2(\kb|\mx-\mz|)^2\right]
\end{align*}
\end{proof}

\begin{rem}\label{Remark-TE}Based on the identified structure \eqref{Structure-TE}, we can observe several properties of the imaging function.
\begin{enumerate}
\item The imaging function $\mathfrak{F}_{\dm}(\mx)$ consists of $J_1(\kb|\mx-\mz_m|)$. Therefore, as in traditional MUSIC-type imaging, the map of $\mathfrak{F}_{\dm}(\mx)$ will contain two peaks of large large magnitude in the neighborhood of ${D}$ and many artifacts with small magnitude.
\item Following \cite[Section 3.5]{P-MUSIC1}, $f_{\fm}(\mx)$ in the TE case can be represented as
\[f_{\fm}(\mx)=\Big(1-J_1(\kb|\mx-\mz|)^2\Big)^{-1/2}.\]
Therefore, similar to TM polarization, the obtained result should be poorer than the result obtained via $f_{\fm}(\mx)$ i.e., a result of poor resolution will be obtained if $N$ is small.
\item Since
\[\lim_{N\to\infty}\mathfrak{F}_{\dm}(\mx)=f_{\fm}(\mx),\]
it can be said that
\begin{multline*}\lim_{N\to\infty}\bigg[(1-C_\mu)^2-2N(1-C_\mu)C_\mu\left(\frac{N}{2}-2\right)J_1(\kb|\mx-\mz|)^2+\frac{N^2}{2}C_\mu^2\left(\frac{N}{2}-2\right)^2J_1(\kb|\mx-\mz|)^2\bigg]\\
=1-J_1(\kb|\mx-\mz|)^2.
\end{multline*}
In this case, similar to the TM polarization case, $C_\mu=O(N^{-2})$ but complete form of the $C_\mu$ is unknown. Notice that if $C_\mu$ satisfies
\[N\left(\frac{N}{2}-2\right)C_\mu=(2+\sqrt2)(1-C_\mu)\quad\text{or equivalently}\quad C_\mu=\frac{4+2\sqrt{2}}{N^2-4N+4+2\sqrt2}=O\left(\frac{1}{N^2}\right),\]
then similar to the TM polarization case, $\mathfrak{F}_{\dm}(\mx)$ can be written as
\[\mathfrak{F}_{\dm}(\mx)=\frac{1}{1-C_\mu}\Big(1-J_1(\kb|\mx-\mz|)^2\Big)^{-1/2}=\left(\frac{N^2-4N+4+2\sqrt2}{N^2-4N}\right)\Big(1-J_1(\kb|\mx-\mz|)^2\Big)^{-1/2}.\]
\item Opposite to the Corollary \ref{CorUnique}, the unique determination in TE polarization cannot be guaranteed via the map of $\mathfrak{F}_{\dm}(\mx)$.
\end{enumerate}
\end{rem}

\section{Simulation results with synthetic and experimental data}\label{sec:5}
In order to validate the results in Theorems \ref{Theorem-TM} and \ref{Theorem-TE}, some results of numerical simulation are exhibited. Motivated by the simulation configuration \cite{BS}, we set the far-field pattern data were obtained in the anechoic chamber with vacuum permittivity $\epsb=\SI{8.854e-12}{\farad/\meter}$ and permeability $\mub=\SI{1.257e-6}{\henry/\meter}$. The ROI $\Omega$ was selected as a square region $(-\SI{0.1}{\meter},\SI{0.1}{\meter})\times(-\SI{0.1}{\meter},\SI{0.1}{\meter})$ and three circular inhomogeneities $D_m\subset\Omega$, $m=1,2,3$, with same radii $\alpha_m=\SI{0.01}{\meter}$, permittivities $\eps_m$, permeabilities $\mu_m$, and locations $\mz_1=(\SI{0.07}{\meter},\SI{0.05}{\meter})$, $\mz_2=(-\SI{0.07}{\meter},\SI{0.00}{\meter})$, and $\mz_3=(\SI{0.02}{\meter},-\SI{0.05}{\meter})$ were chosen. With this configuration, the far-field pattern data $u_{\infty}(\vv_j,\vt_l)$ of $\mathbb{K}$ were generated by solving the Foldy-Lax formulation (see \cite{HSZ4} for instance). After the generation of the far-field pattern, $\SI{20}{\deci\bel}$ white Gaussian random noise was added to the unperturbed data.

\begin{ex}[Dielectric permittivity contrast case]\label{Example-eps}
Figure \ref{Result-eps-1} shows the maps of $\mathfrak{F}_{\tm}(\mx)$ and $\mathfrak{F}_{\dm}(\mx)$ with $f=\SI{1}{\giga\hertz}$ for $N=12$ and $N=36$ directions when $\eps_m=5\epsb$ and $\mu_m\equiv\mub$. Based on this result, the existence of three inhomogeneities $D_m$ can be recognized through the maps of $\mathfrak{F}_{\tm}(\mx)$ and $\mathfrak{F}_{\dm}(\mx)$ but retrieved location of $D_3$ through the map of $\mathfrak{F}_{\dm}(\mx)$ is not accurate.

Figure \ref{Result-eps-2} shows the maps of $\mathfrak{F}_{\tm}(\mx)$ and $\mathfrak{F}_{\dm}(\mx)$ with $f=\SI{2}{\giga\hertz}$ for $N=12$ and $N=36$ directions. Opposite to the imaging result at $f=\SI{1}{\giga\hertz}$, it is possible to retrieve the location of every inhomogeneities very accurately through the map of $\mathfrak{F}_{\dm}(\mx)$. However, although retrieved location of $D_3$ with $N=36$ is more accurate than the one with $N=12$, exact location of $D_3$ cannot be retrieved still.

Based on the simulation results, we can conclude that it will be very difficult to identify exact location of inhomogeneities without diagonal elements of MSR matrix when total number of directions $N$ is small. This is the reason why the condition of sufficiently large $N$ is needed in Theorem \ref{Theorem-TM}.
\end{ex}

\begin{figure}[h]
\begin{center}
\subfigure[with diagonal elements and $N=12$]{\includegraphics[width=0.25\textwidth]{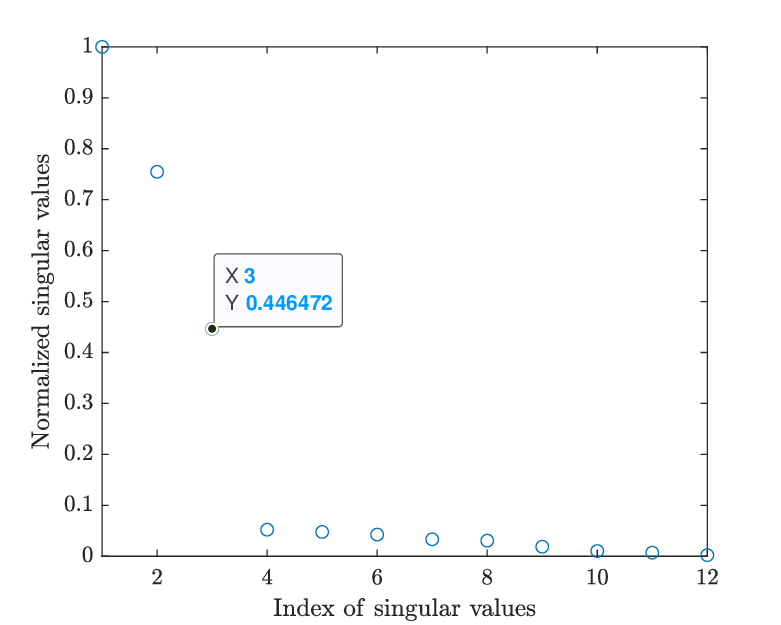}\hfill
\includegraphics[width=0.25\textwidth]{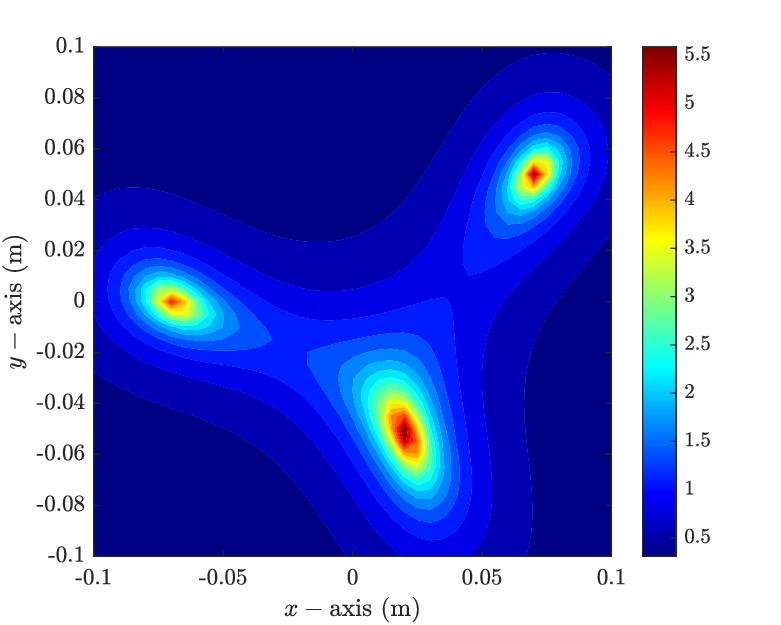}}\hfill
\subfigure[without diagonal elements and $N=12$]{\includegraphics[width=0.25\textwidth]{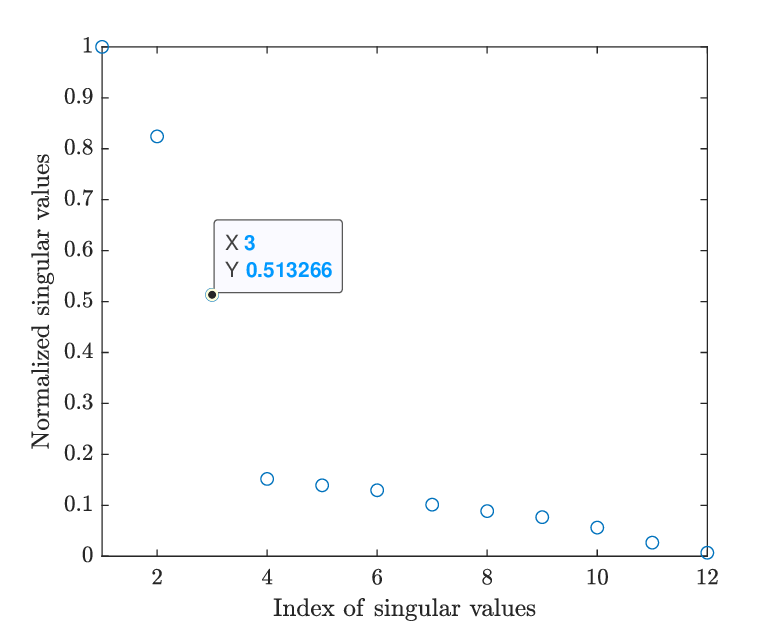}\hfill
\includegraphics[width=0.25\textwidth]{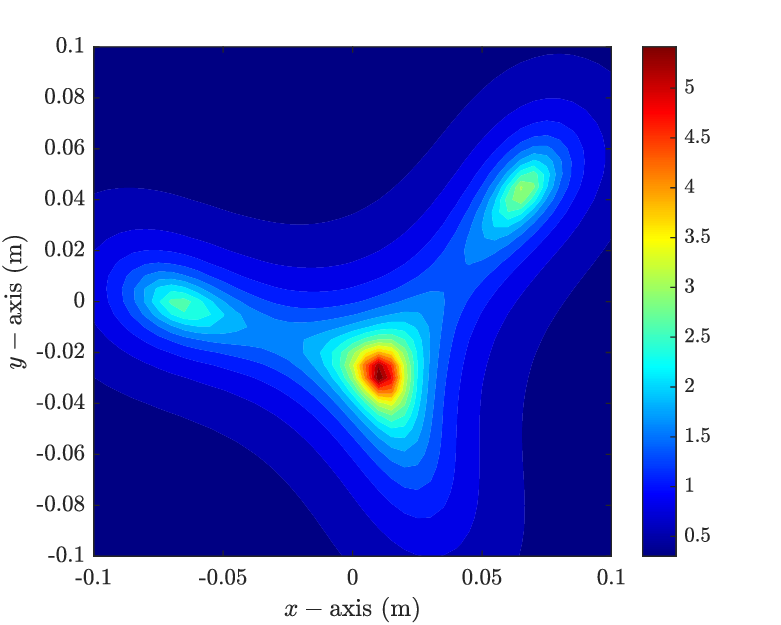}}\\
\subfigure[with diagonal elements and $N=36$]{\includegraphics[width=0.25\textwidth]{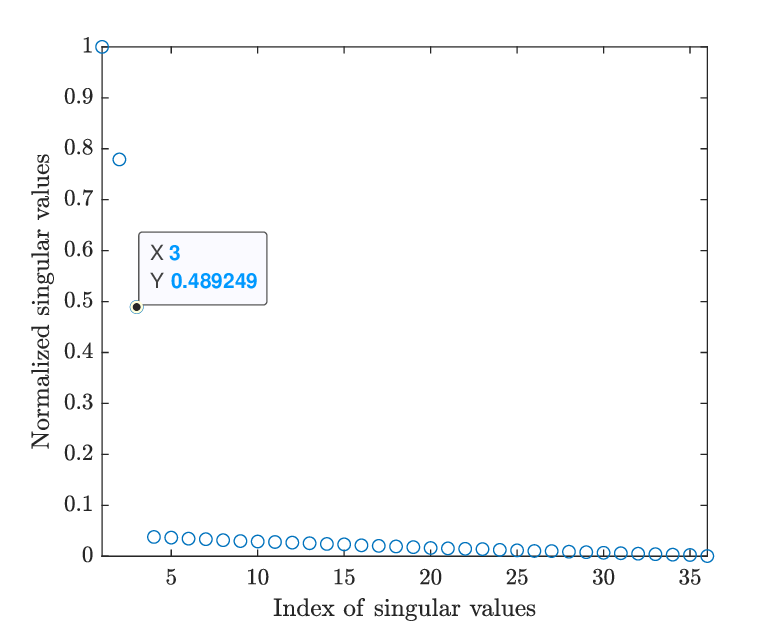}\hfill
\includegraphics[width=0.25\textwidth]{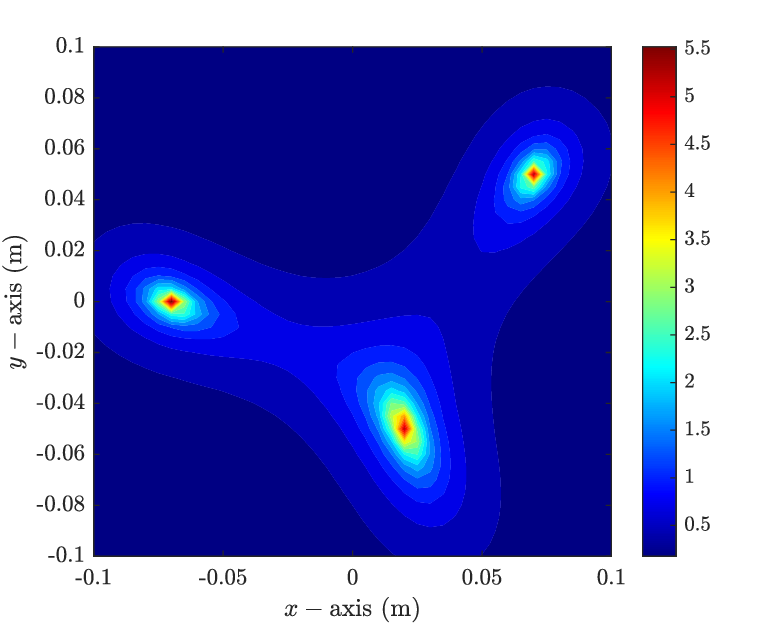}}\hfill
\subfigure[without diagonal elements and $N=36$]{\includegraphics[width=0.25\textwidth]{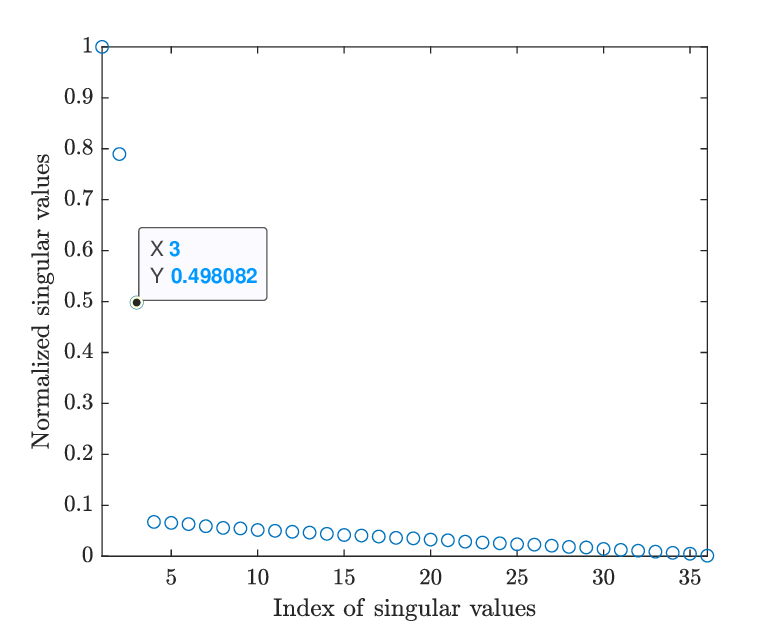}\hfill
\includegraphics[width=0.25\textwidth]{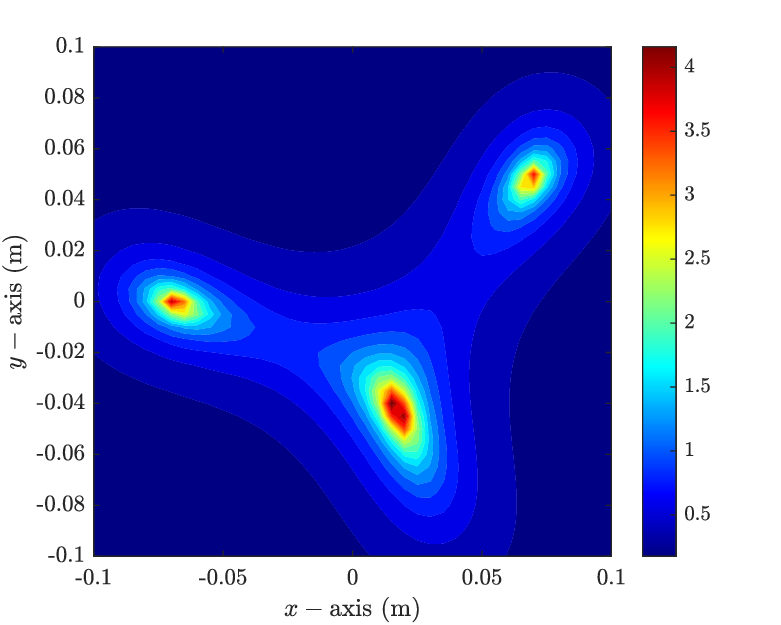}}
\caption{\label{Result-eps-1}(Example \ref{Example-eps}) Distribution of normalized singular values and maps of $\mathfrak{F}_{\tm}(\mx)$ and $\mathfrak{F}_{\dm}(\mx)$ at $f=\SI{1}{\giga\hertz}$.}
\end{center}
\end{figure}

\begin{figure}[h]
\begin{center}
\subfigure[with diagonal elements and $N=12$]{\includegraphics[width=0.25\textwidth]{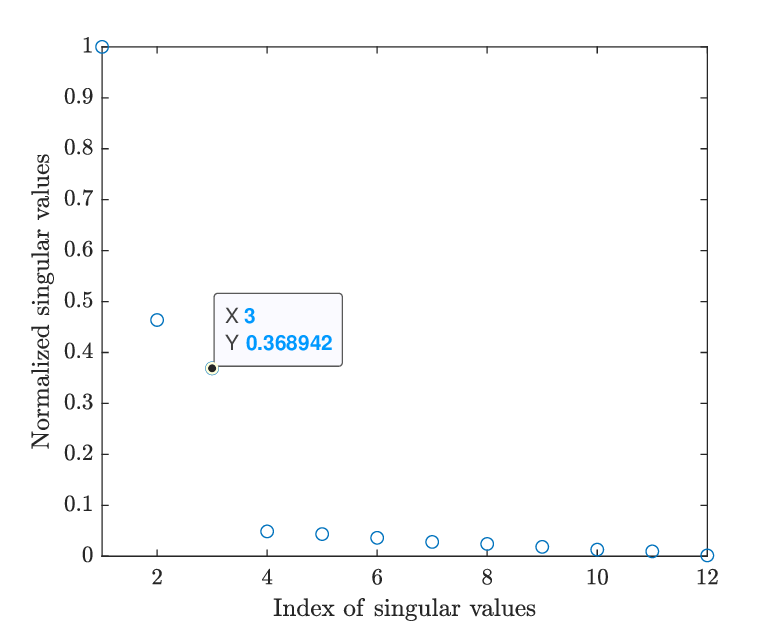}\hfill
\includegraphics[width=0.25\textwidth]{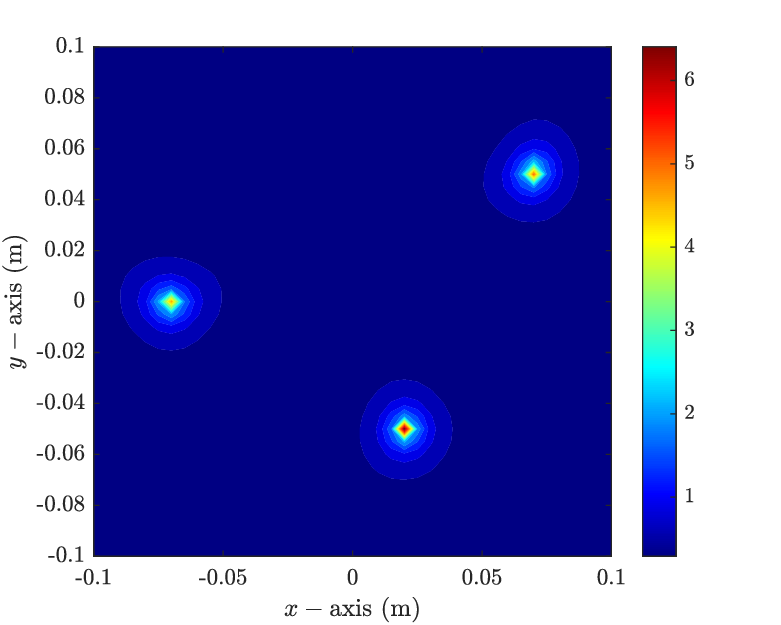}}\hfill
\subfigure[without diagonal elements and $N=12$]{\includegraphics[width=0.25\textwidth]{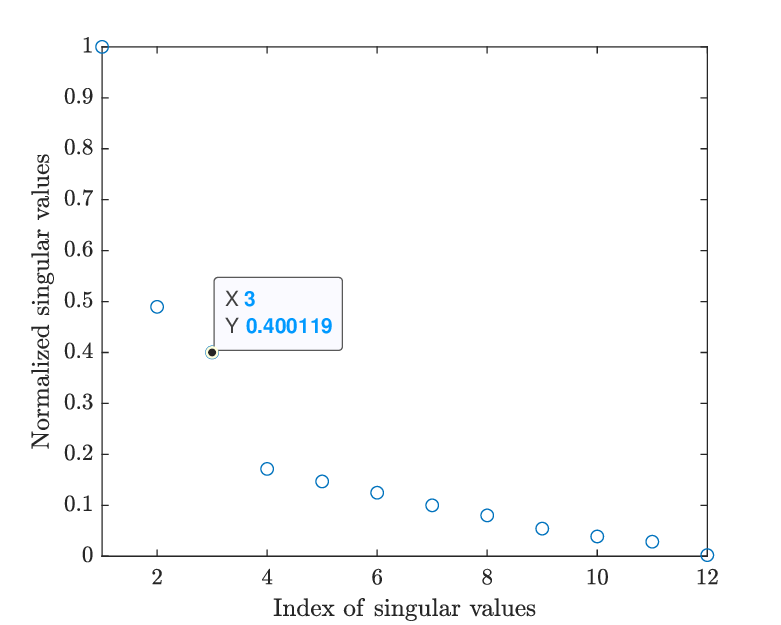}\hfill
\includegraphics[width=0.25\textwidth]{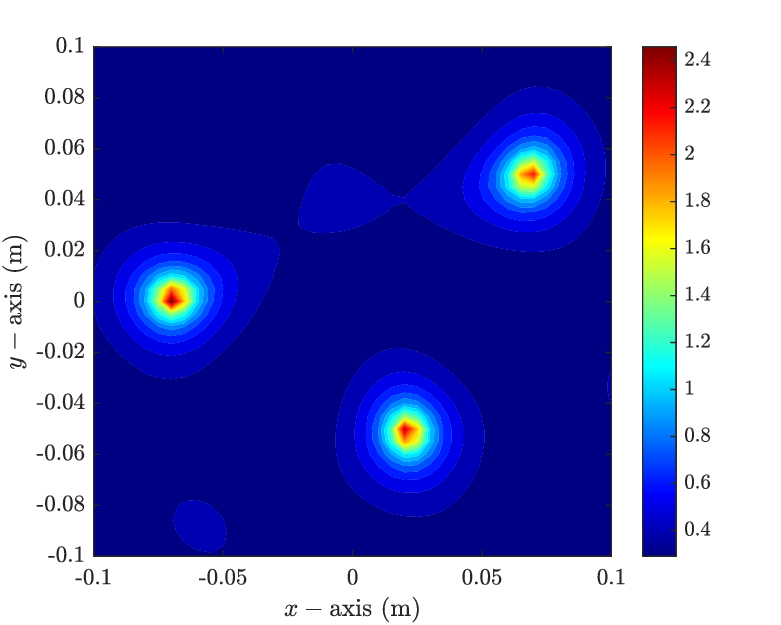}}\\
\subfigure[with diagonal elements and $N=36$]{\includegraphics[width=0.25\textwidth]{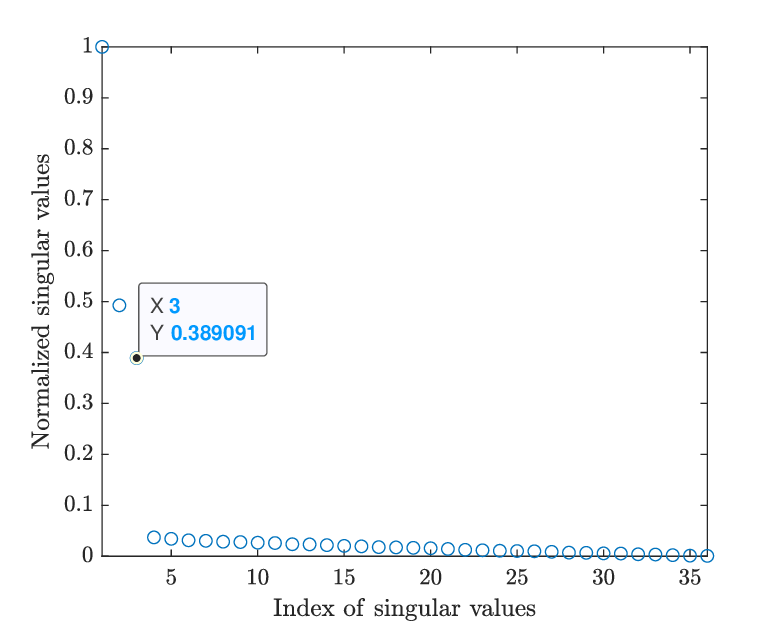}\hfill
\includegraphics[width=0.25\textwidth]{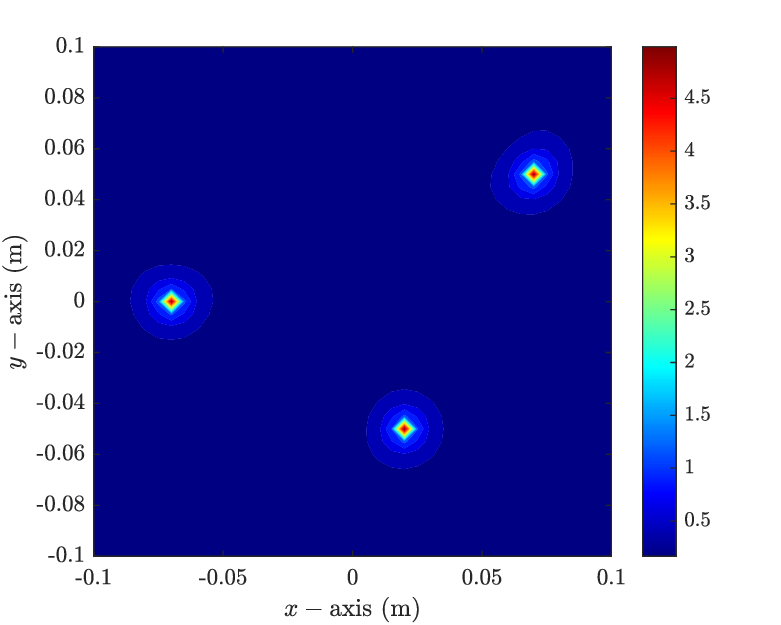}}\hfill
\subfigure[without diagonal elements and $N=36$]{\includegraphics[width=0.25\textwidth]{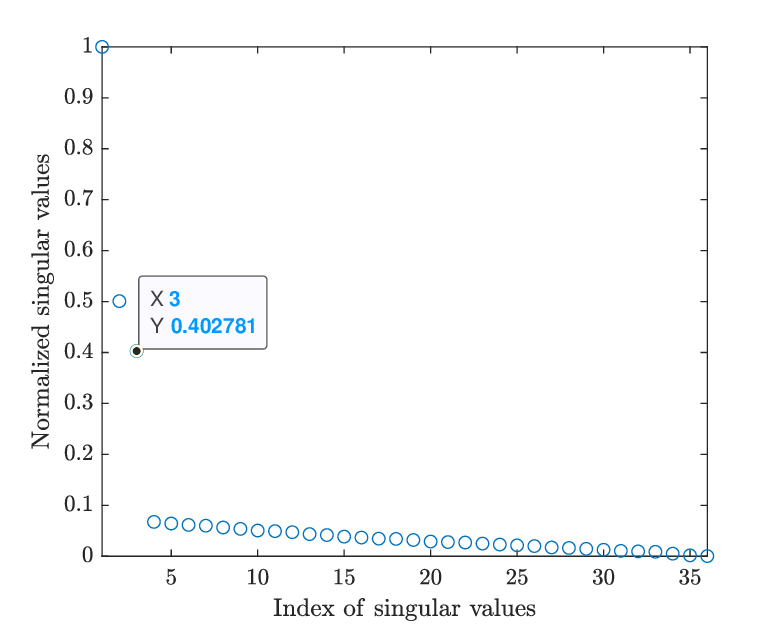}\hfill
\includegraphics[width=0.25\textwidth]{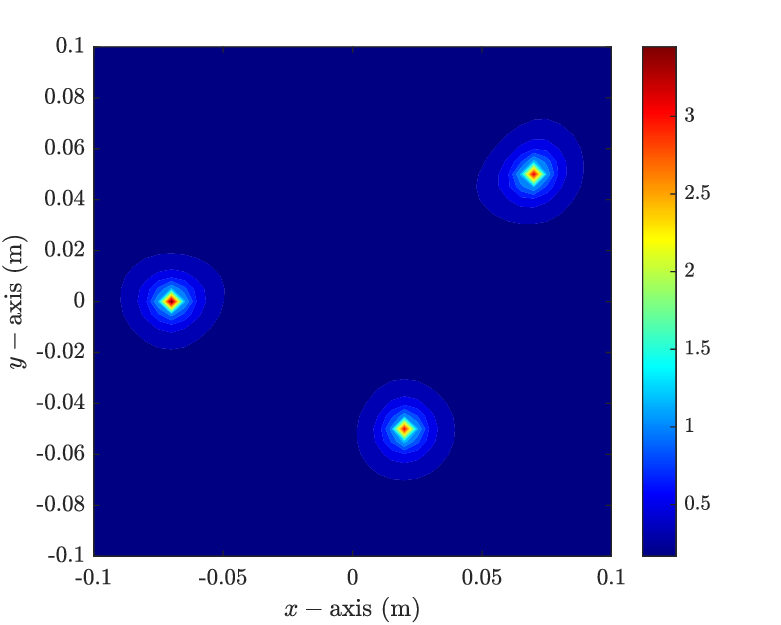}}
\caption{\label{Result-eps-2}(Example \ref{Example-eps}) Distribution of normalized singular values and maps of $\mathfrak{F}_{\tm}(\mx)$ at $f=\SI{2}{\giga\hertz}$.}
\end{center}
\end{figure}

\begin{ex}[Magnetic permeability contrast case]\label{Example-mu}
Figure \ref{Result-mu-1} shows the maps of $\mathfrak{F}_{\tm}(\mx)$ and $\mathfrak{F}_{\dm}(\mx)$ with $f=\SI{2}{\giga\hertz}$ for $N=12$ and $N=36$ directions when $\eps_m\equiv\epsb$ and $\mu_m=5\mub$. Opposite to the Example \ref{Example-eps}, it is impossible to recognize the existence of inhomogeneities with and without diagonal elements of the MSR matrix.

Based on the simulation result with $f=\SI{4}{\giga\hertz}$ with $N=12$ directions, it is still impossible to retrieve the inhomogeneities through the maps of $\mathfrak{F}_{\tm}(\mx)$ and $\mathfrak{F}_{\dm}(\mx)$. Moreover, it is very difficult to discriminate nonzero singular values of the MSR matrix without diagonal elements, refer to Figure \ref{Result-mu-2}. Fortunately, by increasing total number of directions $N=36$, it is very easy to discriminate nonzero singular values and possible to identify the existence of inhomogeneities by regarding the rings in the neighborhood of all $\mz_m\in D_m$. This result supports the Remark \ref{Remark-TE} and we conclude that not only large number of directions $N$ but also high frequency $f$ must be applied to guarantee good imaging results.
\end{ex}

\begin{figure}[h]
\begin{center}
\subfigure[with diagonal elements and $N=12$]{\includegraphics[width=0.25\textwidth]{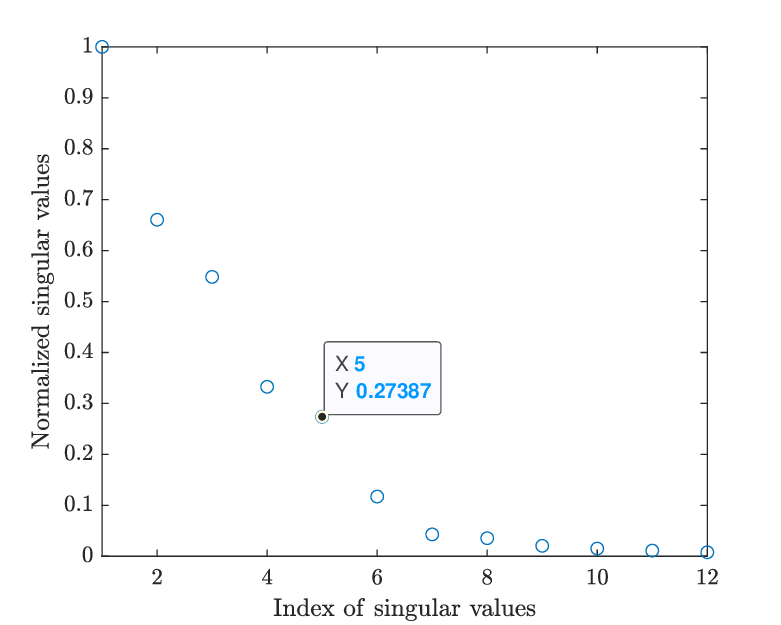}\hfill
\includegraphics[width=0.25\textwidth]{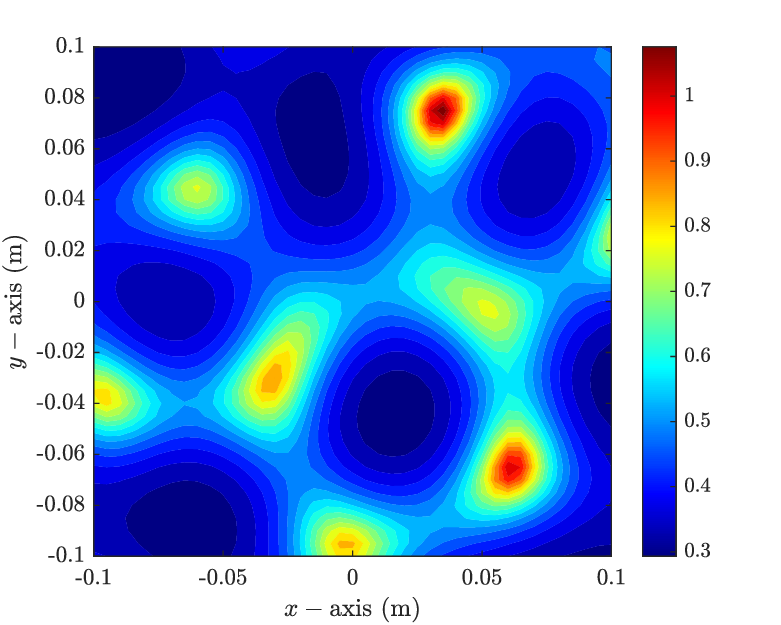}}\hfill
\subfigure[without diagonal elements and $N=12$]{\includegraphics[width=0.25\textwidth]{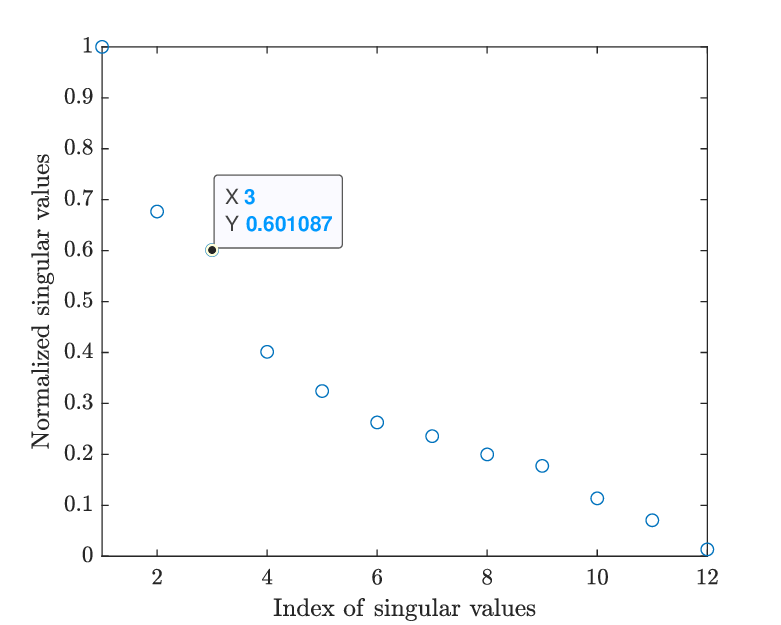}\hfill
\includegraphics[width=0.25\textwidth]{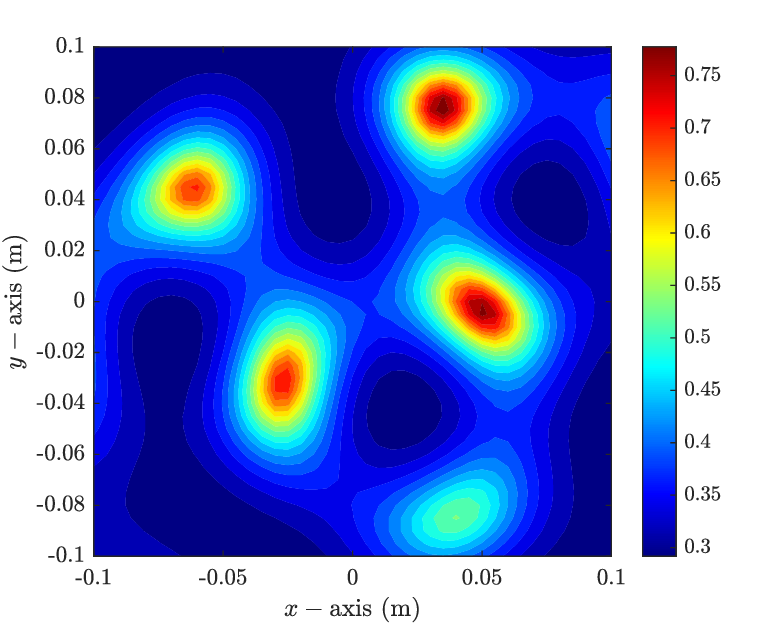}}\\
\subfigure[with diagonal elements and $N=36$]{\includegraphics[width=0.25\textwidth]{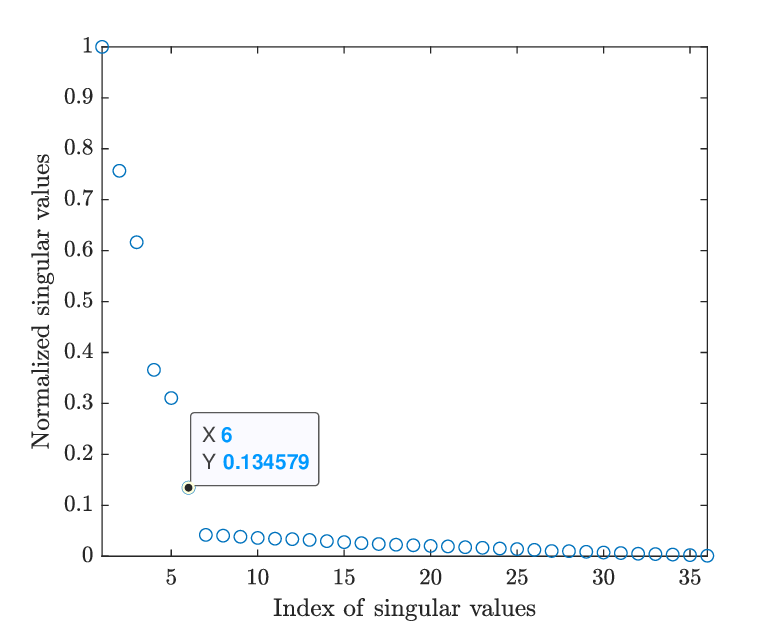}\hfill
\includegraphics[width=0.25\textwidth]{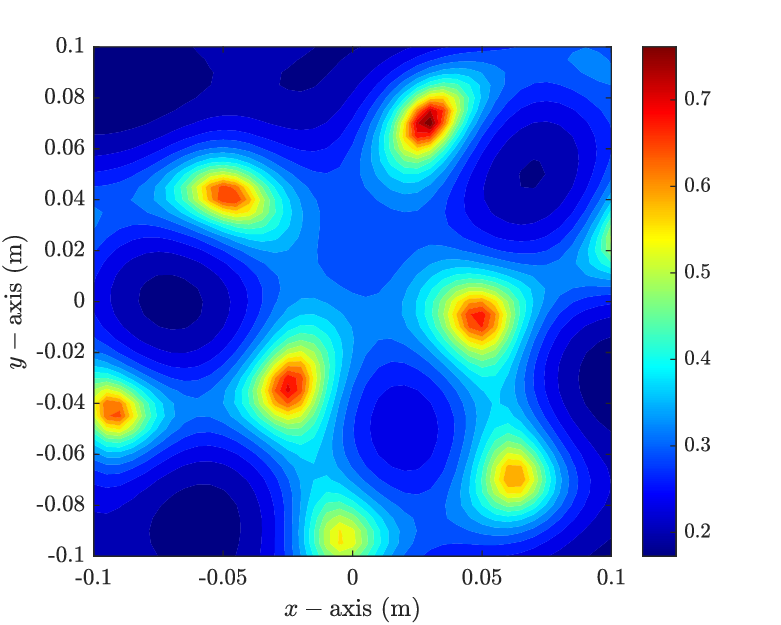}}\hfill
\subfigure[without diagonal elements and $N=36$]{\includegraphics[width=0.25\textwidth]{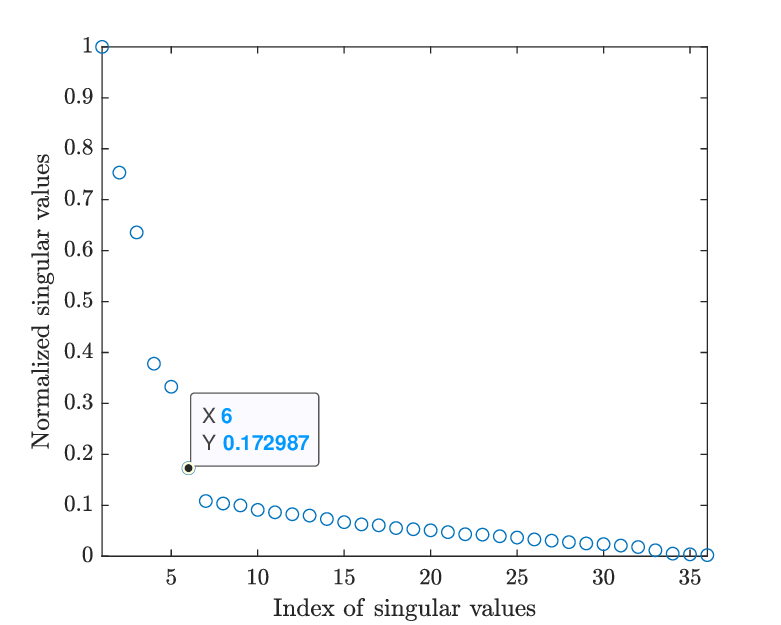}\hfill
\includegraphics[width=0.25\textwidth]{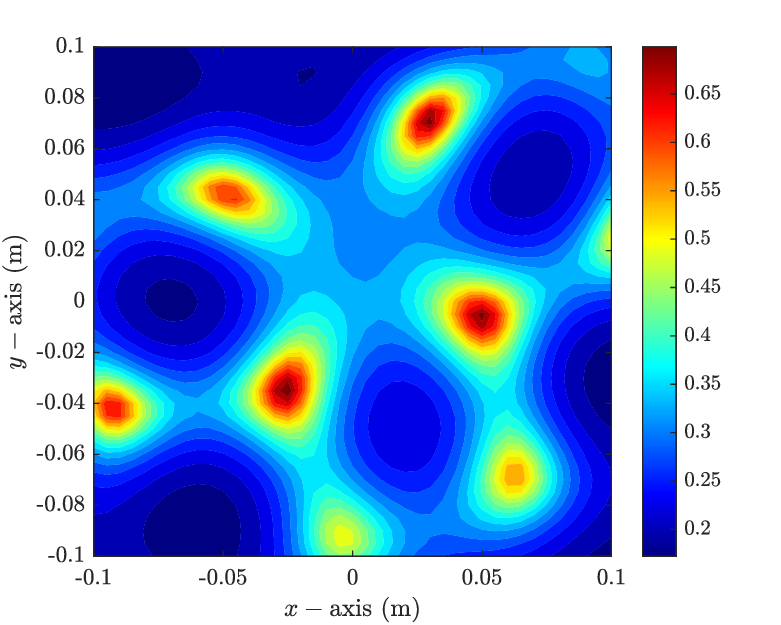}}
\caption{\label{Result-mu-1}(Example \ref{Example-mu}) Distribution of normalized singular values and maps of $\mathfrak{F}_{\tm}(\mx)$ at $f=\SI{2}{\giga\hertz}$.}
\end{center}
\end{figure}

\begin{figure}[h]
\begin{center}
\subfigure[with diagonal elements and $N=12$]{\includegraphics[width=0.25\textwidth]{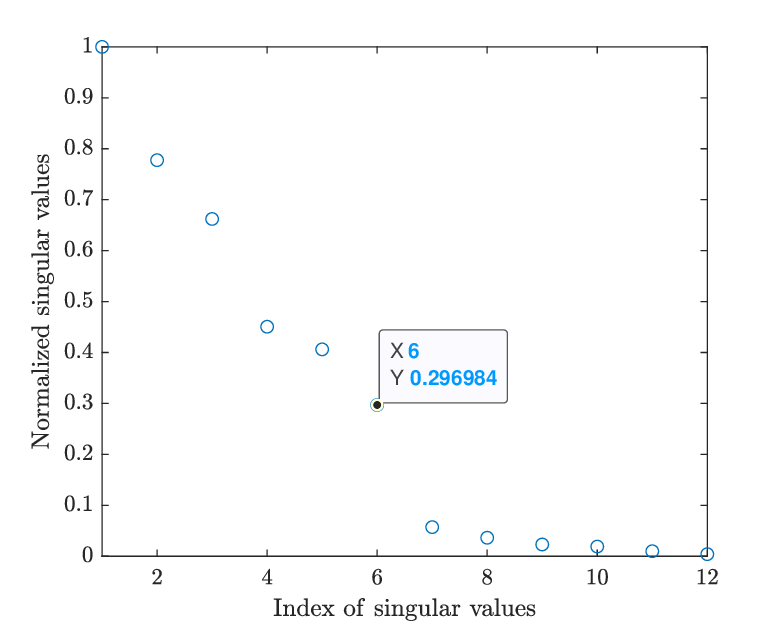}\hfill
\includegraphics[width=0.25\textwidth]{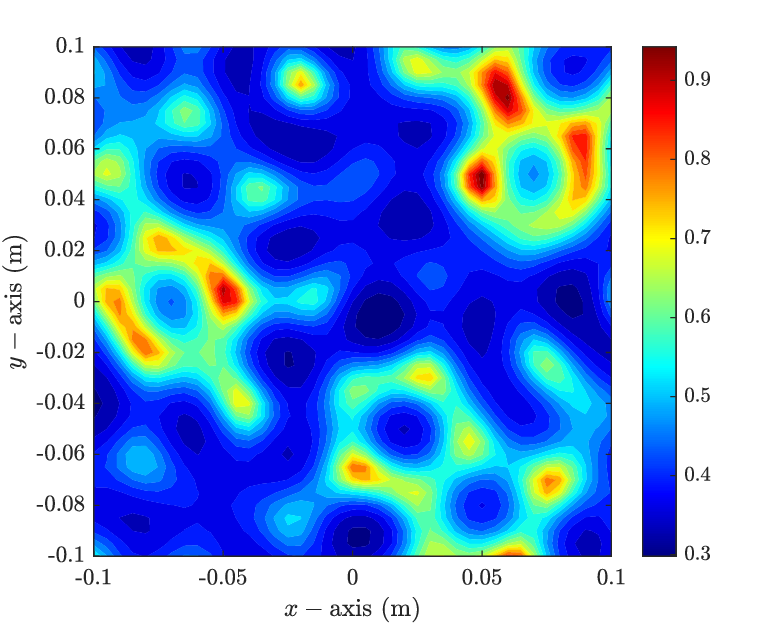}}\hfill
\subfigure[without diagonal elements and $N=12$]{\includegraphics[width=0.25\textwidth]{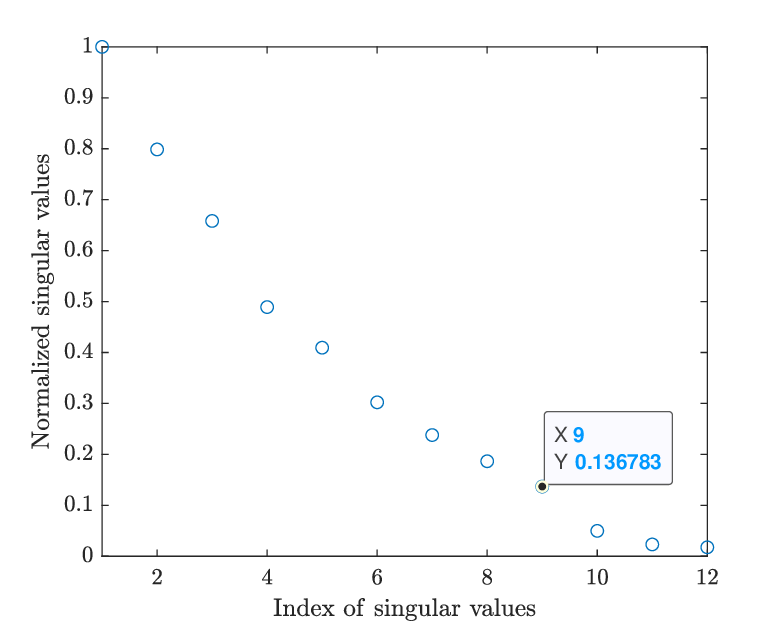}\hfill
\includegraphics[width=0.25\textwidth]{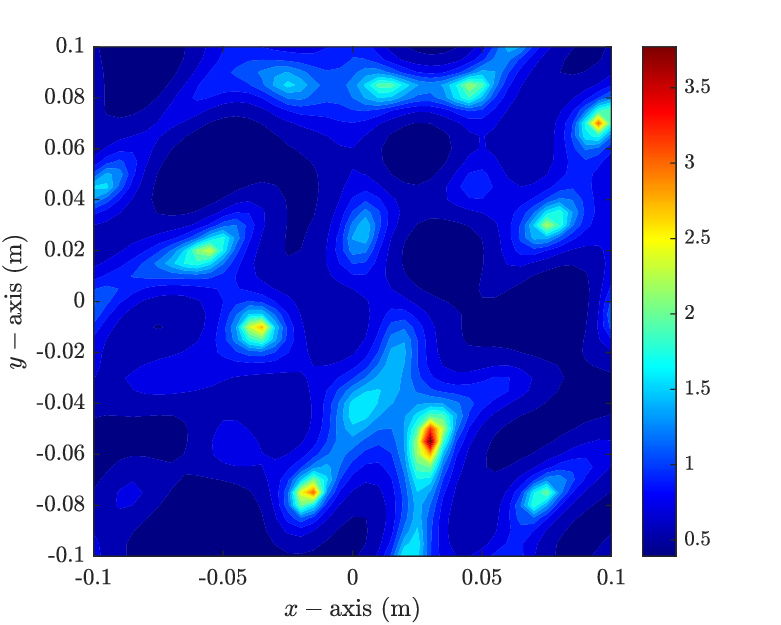}}\\
\subfigure[with diagonal elements and $N=36$]{\includegraphics[width=0.25\textwidth]{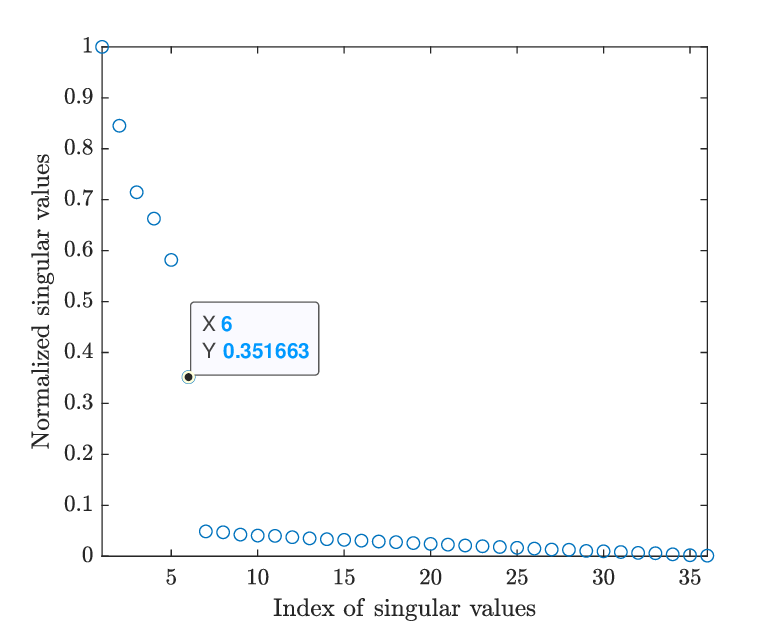}\hfill
\includegraphics[width=0.25\textwidth]{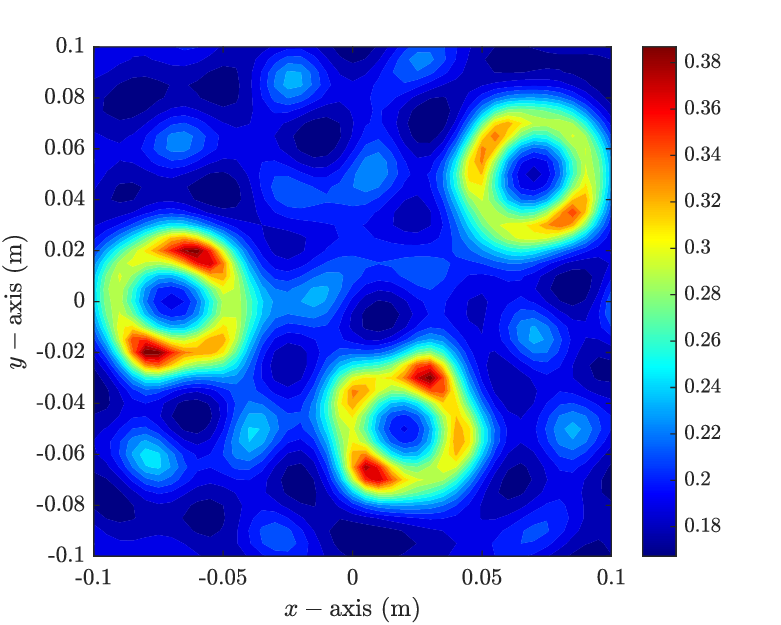}}\hfill
\subfigure[without diagonal elements and $N=36$]{\includegraphics[width=0.25\textwidth]{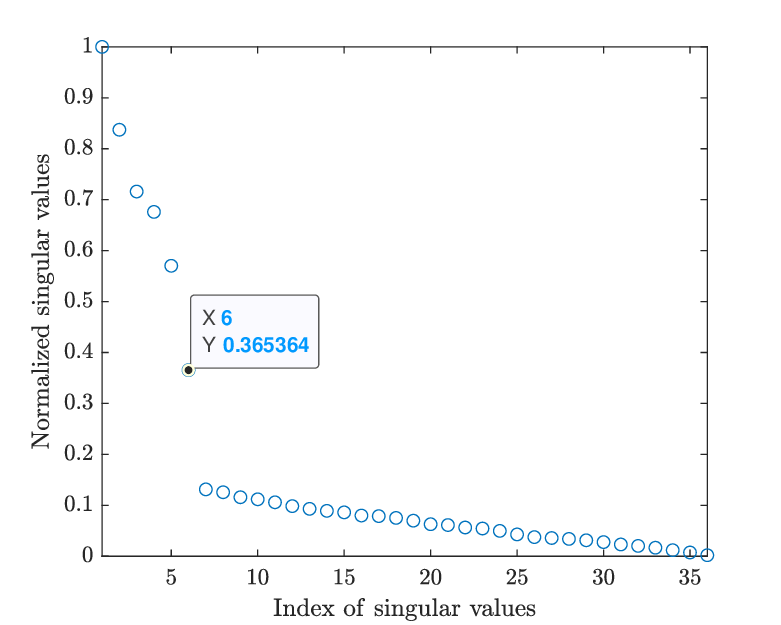}\hfill
\includegraphics[width=0.25\textwidth]{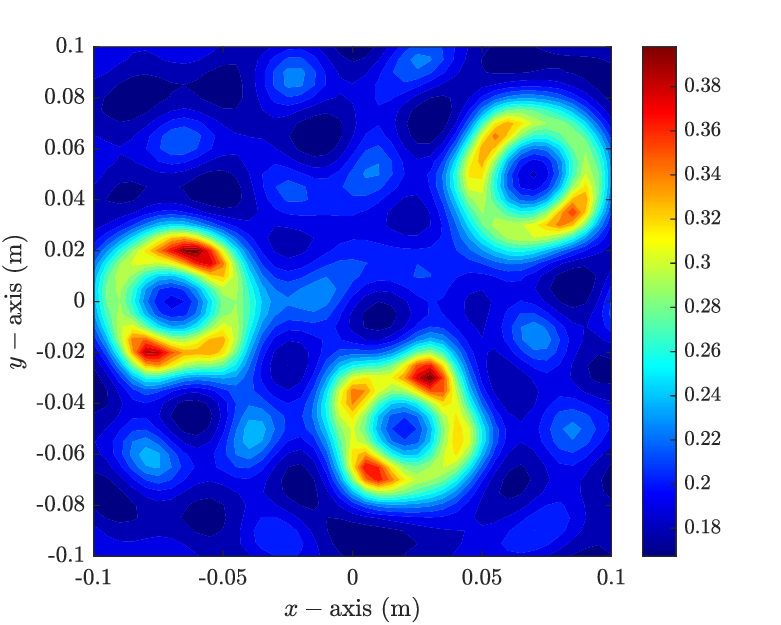}}
\caption{\label{Result-mu-2}(Example \ref{Example-mu}) Distribution of normalized singular values and maps of $\mathfrak{F}_{\tm}(\mx)$ at $f=\SI{4}{\giga\hertz}$.}
\end{center}
\end{figure}

\begin{ex}[Simulation results with experimental data]\label{Example-Fresnel}
Here, we let us consider the simulation results with experimental data \cite{BS}. Following to the simulation configuration in the presence of two dielectric objects (\texttt{twodielTM\underline{ }8f.exp}), the range of receivers is restricted from $\SI{60}{\degree}$ to $\SI{300}{\degree}$, with a step size of $\SI{5}{\degree}$ based on each direction of the transmitters. The transmitters are evenly distributed with step sizes of $\SI{10}{\degree}$ from $\SI{0}{\degree}$ to $\SI{350}{\degree}$. As a result, many elements (totally, $36 \times 23$ measurement data) including the diagonal of the matrix $\mathbb{K}\in\mathbb{C}^{36\times72}$ cannot be measured, We refer to Figure \ref{Matrix}.

Although, the range of the $\mathbb{K}$ is unknown, we consider the application of the MUSIC. To this end, let us perform the SVD
\[\mathbb{K}=\sum_{n=1}^{N}\sigma_n\mU_n\mV_n^*\approx\sum_{n=1}^{N'}\sigma_n\mU_n\mV_n^*.\]
Since $\mathbb{K}$ is non-symmetric, we cannot use the test vector $\mf_{\eps}(\mx)$ of \eqref{testvector-TM} directly. Instead, based on the recent studies \cite{P-SUB16,P-MUSIC7}, we generate projection operators onto the noise subspaces
\[\mathbb{P}_{\noise}=\mathbb{I}_{N'}-\sum_{n=1}^{N'}\mU_n\mU_n^*\quad\text{and}\quad\mathbb{Q}_{\noise}=\mathbb{I}_{N'}-\sum_{n=1}^{N'}\mV_n\mV_n^*,\]
and unit test vectors
\[\mf(\mx)=\frac{1}{\sqrt{36}}\bigg(e^{i\kb\vt_1\cdot\mx},e^{i\kb\vt_2\cdot\mx},\ldots,e^{i\kb\vt_{36}\cdot\mx}\bigg)^T\quad\text{and}\quad
\mg(\mx)=\frac{1}{\sqrt{72}}\bigg(e^{-i\kb\vv_1\cdot\mx},e^{-i\kb\vv_2\cdot\mx},\ldots,e^{-i\kb\vv_{72}\cdot\mx}\bigg)^T.\]
Then, the imaging function can be introduced as
\[\mathfrak{F}_{\dm}(\mz)=\frac{1}{2}\left(\frac{1}{|\mathbb{P}_{\noise}(\mf(\mx))|}+\frac{1}{|\mathbb{Q}_{\noise}(\mg(\mx))|}\right).\]

Based on the imaging results in Figure \ref{Example-Fresnel}, although exact shape of objects cannot be retrieved, the existence and outline shape of objects can be retrieved at $f=\SI{4}{\giga\hertz}$ and $\SI{6}{\giga\hertz}$. However, if one applies low frequency, it will be impossible to recognize the existence of objects (at $f=\SI{1}{\giga\hertz}$) or very difficult to retrieve the outline shape of objects (at $f=\SI{2}{\giga\hertz}$).
\end{ex}

\begin{figure}[h]
\centering
\subfigure[$f=\SI{1}{\giga\hertz}$]{\includegraphics[width=0.25\textwidth]{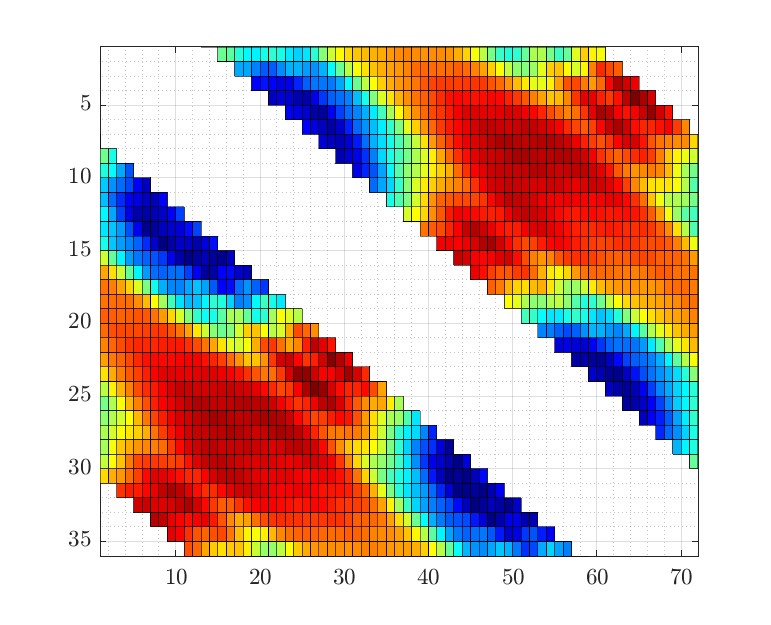}}\hfill
\subfigure[$f=\SI{2}{\giga\hertz}$]{\includegraphics[width=0.25\textwidth]{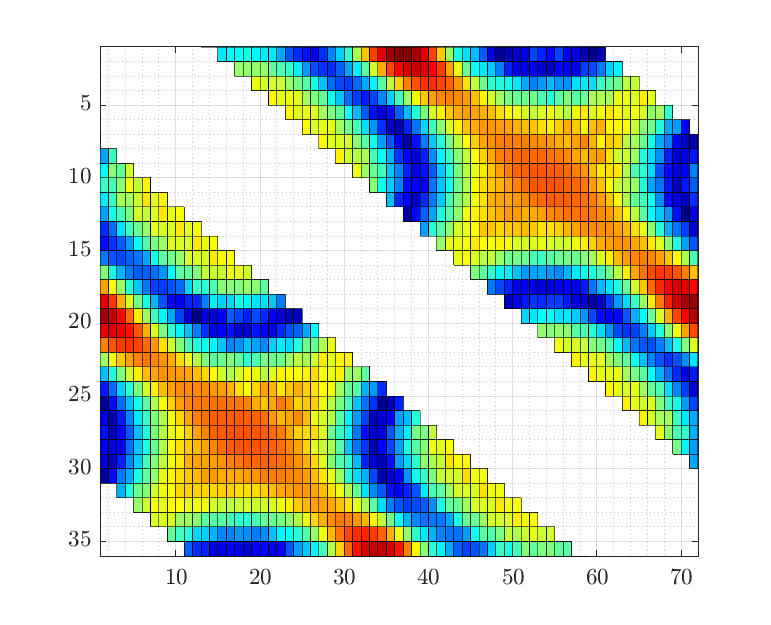}}\hfill
\subfigure[$f=\SI{4}{\giga\hertz}$]{\includegraphics[width=0.25\textwidth]{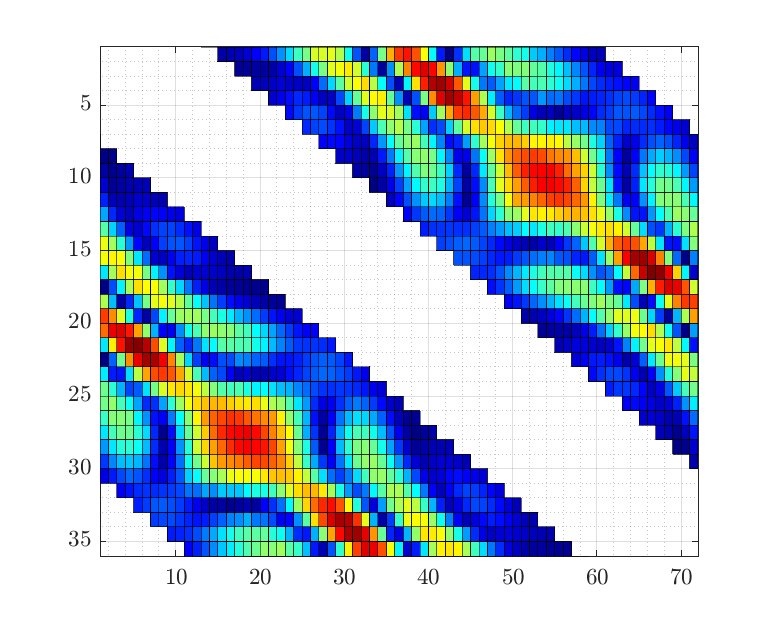}}\hfill
\subfigure[$f=\SI{6}{\giga\hertz}$]{\includegraphics[width=0.25\textwidth]{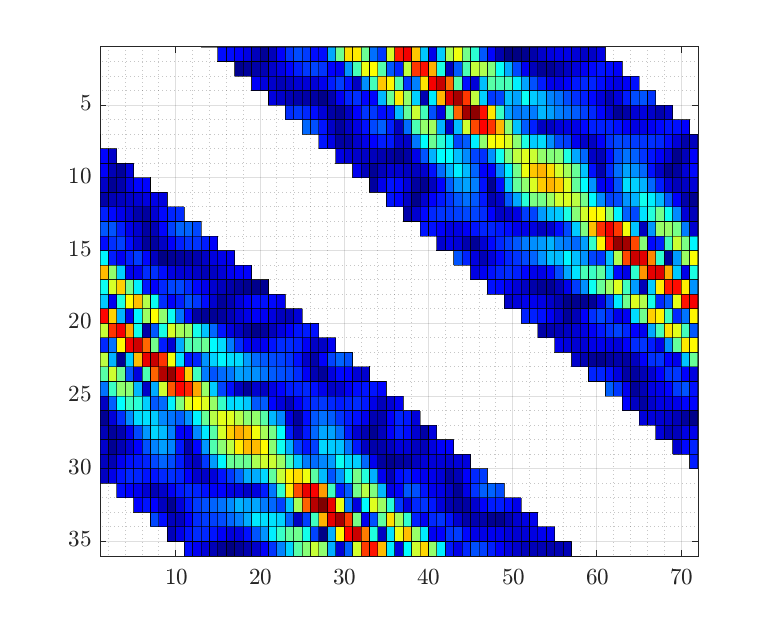}}
\caption{\label{Matrix}(Example \ref{Example-Fresnel}) Visualization of the absolute value of MSR matrix.} 
\end{figure}

\begin{figure}[h]
\begin{center}
\subfigure[$f=\SI{1}{\giga\hertz}$]{\includegraphics[width=0.25\textwidth]{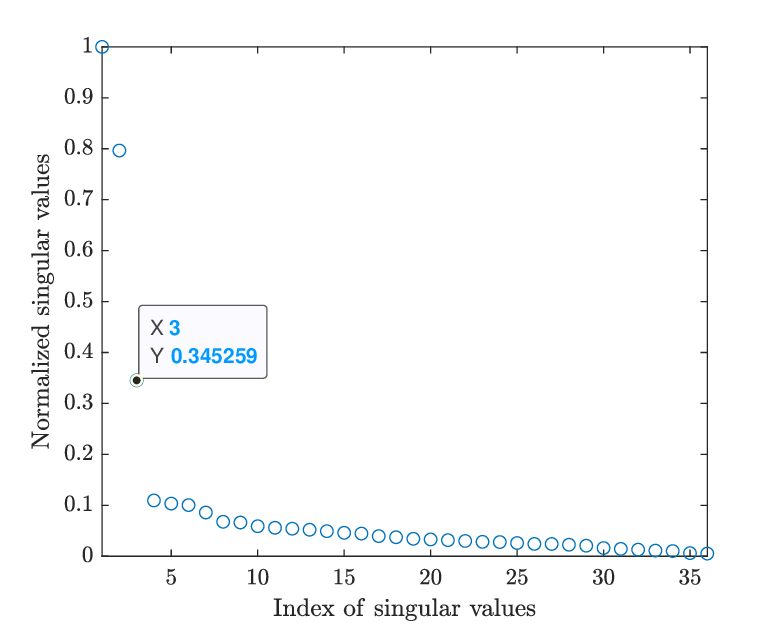}\hfill
\includegraphics[width=0.25\textwidth]{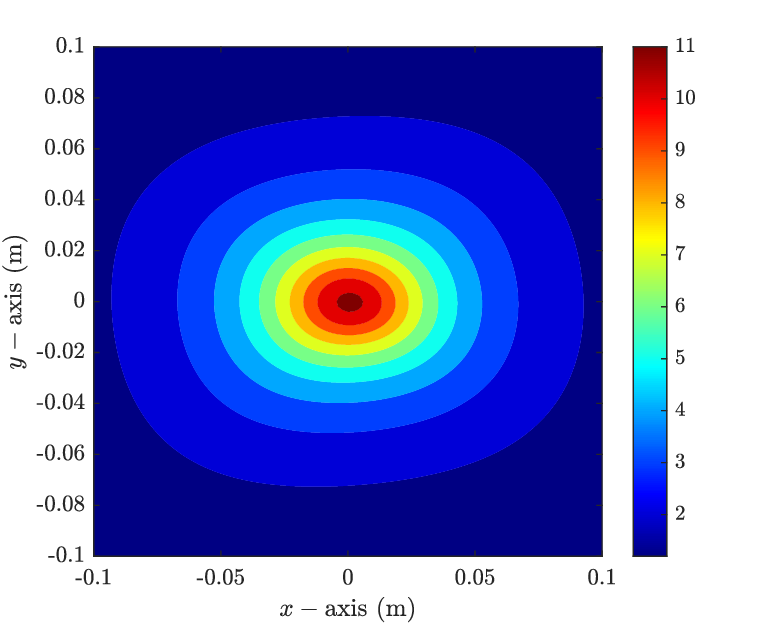}}\hfill
\subfigure[$f=\SI{2}{\giga\hertz}$]{\includegraphics[width=0.25\textwidth]{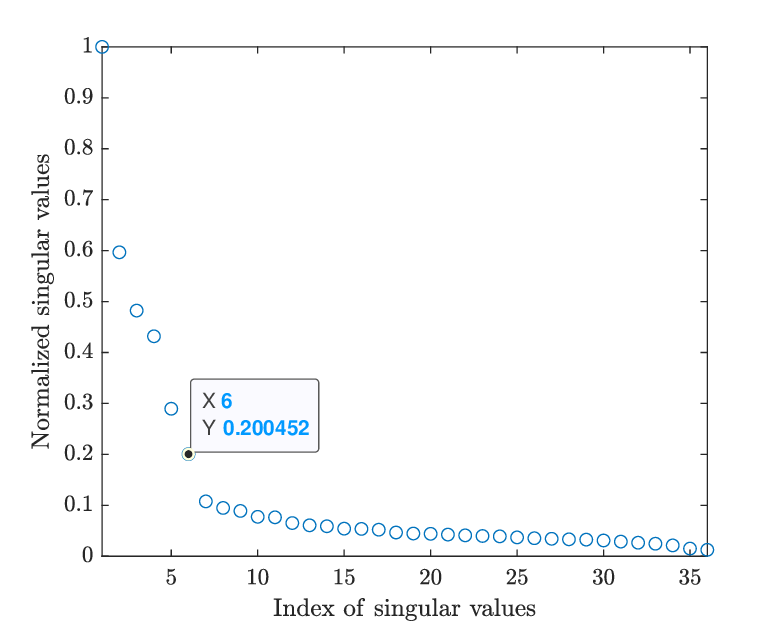}\hfill
\includegraphics[width=0.25\textwidth]{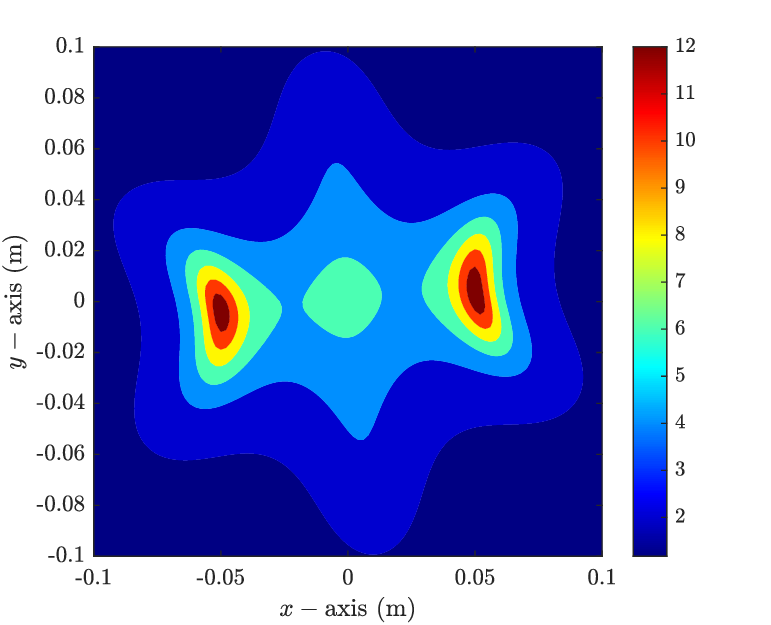}}\\
\subfigure[$f=\SI{4}{\giga\hertz}$]{\includegraphics[width=0.25\textwidth]{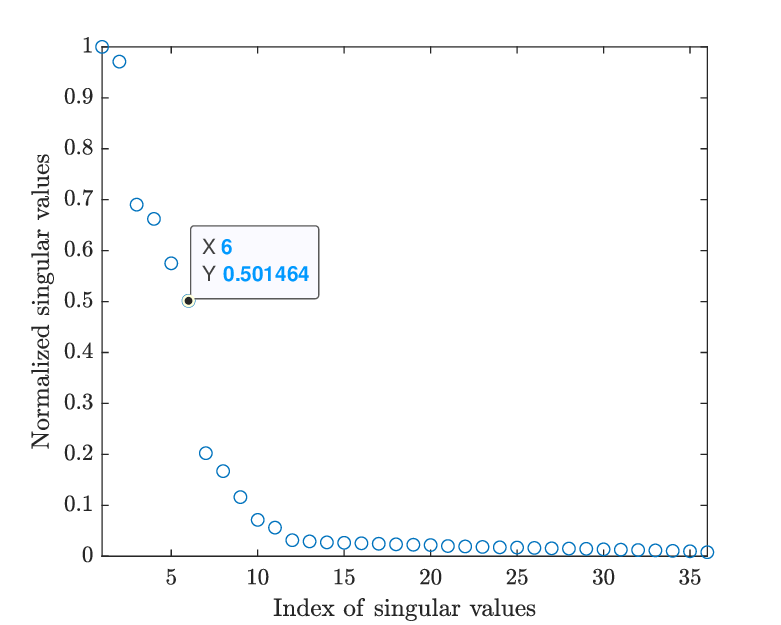}\hfill
\includegraphics[width=0.25\textwidth]{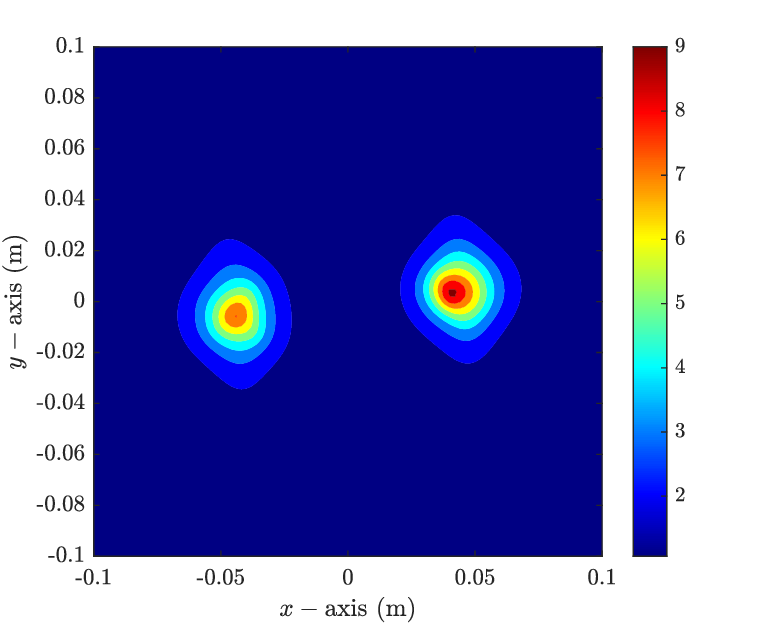}}\hfill
\subfigure[$f=\SI{6}{\giga\hertz}$]{\includegraphics[width=0.25\textwidth]{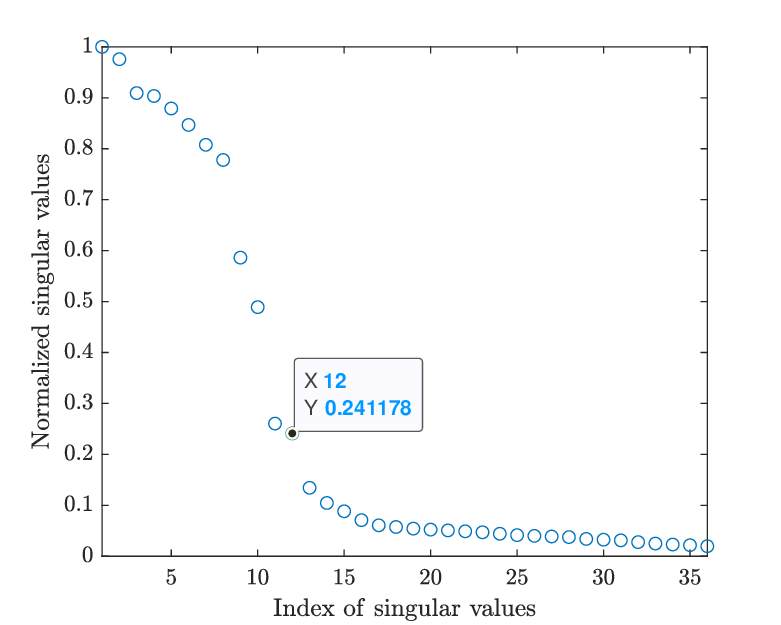}\hfill
\includegraphics[width=0.25\textwidth]{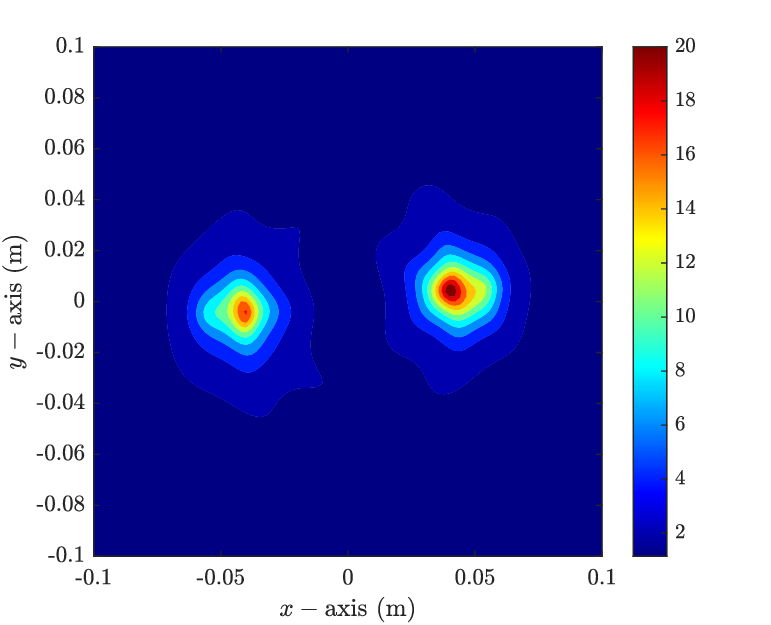}}
\caption{\label{Result-Fresnel}(Example \ref{Example-Fresnel}) Distribution of normalized singular values and maps of $\mathfrak{F}_{\dm}(\mx)$.}
\end{center}
\end{figure}

\section{Conclusion}\label{sec:7}
In this study, we considered the MUSIC algorithm for localizing two-dimensional small inhomogeneity modeled via TM and TE polarization when the diagonal elements of MSR matrix cannot be determined. We investigated a mathematical structure of the imaging functions by establishing a relationship with the Bessel function of order $0$ (TM polarization) and $1$ (TE polarization). Based on the investigated structures, we confirmed that MUSIC can be applied to retrieve location of small inhomogeneity without the diagonal elements of the MSR matrix in both TM and TE polarizations.

Unfortunately, exact expression of $C_\mu$ in Theorem \ref{Theorem-TE} is still unknown. Derivation of exact structure of the imaging function in TE polarization will be an interesting research subject. In this study, the structures were derived in the presence of single inhomogeneity but MUSIC can be applied to the identification of multiple, small inhomogeneities on the basis of simulation results. Extension to the multiple, small inhomogeneities will be the forthcoming work. Finally, extension to the three-dimensional inverse scattering problem will be an interesting research topic.

\section*{Acknowledgments}
This work was supported by the National Research Foundation of Korea (NRF) grant funded by the Korea government (MSIT) (NRF-2020R1A2C1A01005221).

\end{document}